\documentclass[a4paper]{article}

\usepackage{amsmath}%
\usepackage{amssymb}%
\usepackage[latin1]{inputenc}
\usepackage{enumitem}
\usepackage{euscript}

\usepackage{geometry}
\newtheorem{theorem}{Theorem}[section]

\newtheorem{corollary}[theorem]{Corollary}

\newtheorem{definition}[theorem]{Definition}

\newtheorem{lemma}[theorem]{Lemma}

\newtheorem{proposition}[theorem]{Proposition}
\newtheorem{remark}[theorem]{Remark}

\newenvironment{proof}[1][Proof]{\textbf{#1.} }{\ \rule{0.5em}{0.5em}}

\newcommand{\refeqn}[1]{(\ref{#1})}

\newcommand{\cinf}[0]{C^{\infty}}

\newcommand{\matr}[0]{\operatorname{Mat}}

\newcommand{\acc}[1]{\`{#1}}

\begin{document}

\title{{\bf Random transformations and invariance of semimartingales on Lie groups}}
\author{Sergio Albeverio\thanks{Institut f\"ur Angewandte Mathematik and HCM, Rheinische Friedrich-Wilhelms-Universit\"at Bonn, Endenicher Allee 60, Bonn, Germany, \emph{email: albeverio@uni-bonn.de}}, Francesco C. De Vecchi\thanks{Institut f\"ur Angewandte Mathematik and HCM, Rheinische Friedrich-Wilhelms-Universit\"at Bonn, Endenicher Allee 60, Bonn, Germany, \emph{email: francesco.devecchi@uni-bonn.de}}, Paola Morando\thanks{DISAA, Universit\`a degli Studi di Milano, via Celoria 2, Milano, Italy, \emph{email: paola.morando@unimi.it}} and Stefania Ugolini\thanks{Dipartimento di Matematica, Universit\`a degli Studi di Milano, via Saldini 50, Milano, Italy, \emph{email: stefania.ugolini@unimi.it}}}

\maketitle

\begin{abstract}
Invariance properties of semimartingales on Lie groups under a family of random transformations are defined and investigated, generalizing the random rotations of the Brownian motion. A necessary and sufficient explicit condition characterizing semimartingales with this kind of invariance
is given in terms of their stochastic characteristics. Non trivial examples of symmetric semimartingales are provided and applications of this concept to stochastic analysis are discussed.
\end{abstract}



\section{Introduction}

For deterministic differential equations the study of transformations acting on the underlying time and space variables and the induced transformations in the space of solutions is a well-developed subject of research.
This study also includes themes like symmetries and invariance properties for the solutions.

In the 19th century, particularly since the inception of the concept of group (Abel, Galois) and in particular of continuous groups of transformations (Lie, Klein) the study of transformations of differential equations has lead to a reduction theory, used to bring general equations to simpler forms.
Moreover, the finding of invariants under the group of transformations permits to reduce the number of relevant variables and even in certain cases to arrive at concrete solutions.
This is successfully exploited in connection with classical mechanics and dynamical systems, described by systems of ordinary differential equations, see, e.\,g., \cite{Gaeta1994,Olver1993,Stephani1989}.

For the case where the state space for the solution process is a manifold, or more specifically a Lie group (like, e.\,g., in the classical problem of the motion of a top), geometric mechanics provides a natural setting, see, e.\,g., \cite{Holm2009}.
Variational principles and the well known relation between symmetries and invariants (Noether's theorem) provide other important methods for the study of such deterministic systems.

A corresponding theory for equations with random terms is still much less developed, despite the well-developed theory of (It\^o or Stratonovich type) stochastic differential equations.
Recently, however, an increasing interest has been given to this topic, and to the study of stochastic processes and their transformation and symmetry properties, see, e.\,g., \cite{Albeverio2017,Albeverio1998,Cicogna1999,DeVecchithesis,DMU2,DMU1,Gaeta2018,Kallenberg2005,Cami2009,Misawa1994}.

For the case of processes with values in $\mathbb{R}^n$ and corresponding to the stochastic equations, the study of transformations and symmetries in analogy with the deterministic theory can be found, e.g., in \cite{DeVecchithesis,Liao1992,Liao2016} for stochastic equation with driving Brownian motion and, e.g., in \cite{Applebaum2004} for more general driving L\'evy processes.
For the case where the processes take values in Lie groups and symmetric spaces see, e.\,g., \cite{Applebaum2000,Applebaum2015,Feinsilver1978,Gangolli1964,Hunt1956,Liao2004,Liao2014}.

More generally Markov processes and their symmetries have been studied in a number of papers, see, e.\,g., \cite{DMU2,DMU1,Glover1991,Glover1990,Liao1992}.
Let us also mention that symmetry with respect to deterministic group of permutations has been exploited in de Finetti characterization of exchangeable processes, while processes which are symmetric with respect to groups of random transformations have been studied in \cite{Kallenberg2005} and references therein.\\

In the probabilistic setting there is a special interest in looking for transformations that transform both the coefficients and the underlying noise. Especially on Lie groups it seems natural to consider the general setting of semimartingales, thus extending work on some special processes (e.\,g., Brownian motions and their subordinates on Lie groups).
This is the subject of the present paper, where we introduce a new concept of invariance for semimartingales defined on Lie groups which we call \emph{gauge symmetry}. The prototypical example of gauge symmetry is the well known invariance of Brownian motion with respect to random rotations. More precisely, if $W$ is an $n$ dimensional Brownian motion and $\bold{B}:\Omega \times \mathbb{R}_+ \rightarrow O(n)$ is a predictable stochastic process taking values in the Lie group $O(n)$ of orthogonal $n \times n$ matrices, by the L\'evy characterization of Brownian motion, we have that $W'_t=\int_0^t{\bold{B}_s \cdot dW_s}$ is a new $n$ dimensional Brownian motion. This property has some interesting consequences for the Markovian Brownian-motion-driven SDEs. In particular, the law of the solution $X_t \in \mathbb{R}^n$ to a SDE $(\mu,\sigma)$ driven by the Brownian motion $W$ is not uniquely characterized by the mathematical objects $(\mu,\sigma)$ and $W$. Indeed $X$ is also a solution to the SDE $(\mu, \sigma \cdot B)$ driven by  $W'_t=\int_0^t{B^{-1}(X_s,s)\cdot dW_s}$, where $B:\mathbb{R}^m \times \mathbb{R}_+ \rightarrow O(n)$ is any measurable function. This means that weak solutions to a Brownian motion driven SDE are not identified by the coefficients $(\mu,\sigma)$ of the SDE but by the generator $L=\mu \cdot \partial+\frac{1}{2}\sigma \cdot \sigma^T \cdot \partial^2$, which is invariant with respect to random rotations. In this work we extend this invariance property from Brownian motion to general semimartingales taking values on a finite dimensional Lie group.\\

The first step is the introduction of the notion of \emph{geometrical SDE}, inspired by the works of Serge Cohen (see \cite{Cohen1995,Cohen1996} and also \cite{Applebaum1997}). Given a finite dimensional manifold $M$, a finite dimensional Lie group $N$ and a topological space $\mathcal{K}$, a geometrical SDE driven by a semimartingale on $N$ is described by a smooth map $\Psi_{\cdot}:M \times N \times \mathcal{K} \rightarrow M$ such that $\Psi_k(x,1_N)=x$ and by a predictable locally bounded process $K_t$ defined on $\mathcal{K}$. The map $\Psi_k$, in some way, describes the jumps of the solution process $X$ with respect to the jumps of the driving process $Z$. Indeed, if we set$\Delta Z_t=Z_t \cdot Z^{-1}_{t_-}$, we have that $X_t=\Psi(X_{t_{-}},\Delta Z_t)$. If the manifold $M$ coincides with the Lie group $N$, we can restrict our attention to \emph{right invariant geometrical SDEs}, i.e. SDEs of the form $\Psi_k(\tilde{z},z)=\Xi_k(z) \cdot \tilde{z}$, where $\Xi_{\cdot}:N \times \mathcal{K} \rightarrow N$ satisfies $\Xi_k(1_N)=1_N$. This kind of SDEs, transforming the semimartingale $Z$ on $N$ in a new semimartingale $\tilde{Z}$ on $N$, which is the unique (strong) solution to the SDE $d\tilde{Z}_t=\Xi_{K_t}(dZ_t)$ such that $\tilde{Z}_0=1_N$, provides the set of \emph{random transformations}. The relation between random transformations and geometrical SDEs extends the associative property of the It\^o integral: if $X$ is a solution to the SDE $\Psi_{K_t}(x,z)$ driven by $\tilde{Z}$, then $X$ is a solution to the SDE $\Psi_{K_t}(x,\Xi_{K'_t}(z))$ driven by $Z$. Moreover, if $\mathcal{K}=\mathcal{G}$ is a topological group and $\Xi_g$ is a group action on $N$, the set of random transformations forms a group where the inverse is given by $dZ_t=\Xi_{G^{-1}_t}(d\tilde{Z}_t)$.\\
In Section \ref{subsection_comparison}, we provide a comparison between geometrical SDEs and other notions of SDEs driven by c\acc{a}dl\acc{a}g semimartingales commonly found in the literature, such as Marcus-type SDEs (see \cite{Protter1995,Marcus1981}) and their generalizations (for example the theory of stochastic flows with jumps on manifolds \cite{Applebaum2001,Kunita1999} and rough paths theory for processes with jumps \cite{Friz2017(2),Friz2017}), SDEs driven by L\'evy processes through Poisson measures (see \cite{Applebaum2004,Kunita2004}), or more general random measures (see \cite{Bichteler2002}).\\

Once we have introduced random transformations, we can consider a semimartingale $Z_t \in N$ admitting a \emph{gauge symmetry group $\mathcal{G}$ with action $\Xi_g$}  as the family of semimartingales for which the random transformations of the form $\Xi_{G_t}(dZ_t)$ preserve the law of $Z$ for any predictable locally bounded process $G_t \in \mathcal{G}$.\\
Furthermore, in order to make our notion of invariance effective, we provide a simple method for verifying whether a given semimartingale admits a gauge symmetry group, generalizing the role played by the L\'evy characterization of Brownian motion in the proof of random rotation invariance of Brownian motion. Indeed Theorem \ref{theorem_characteristic2} gives a necessary and sufficient condition for the presence of a gauge symmetry group in terms of the \emph{characteristic triplet $(b,A,\nu)$} of the semimartingale $Z$ on the Lie group $N$. In the case of a L\'evy process on $N$, where $(b,A,\nu)$ are deterministic, this general condition can be reduced to a simpler deterministic one (see Theorem \ref{theorem_characteristic3}). Furthermore, we extend the previous deterministic characterization of gauge invariant L\'evy processes to some special cases of non-Markovian semimartingales (see Theorem \ref{theorem_characteristic4}), providing also some examples of non-Markovian gauge symmetric semimartingales (see Section \ref{subsection_non_markovian}).\\

Gauge symmetries generalize in two different directions the results on symmetries of stochastic processes appearing in the previous literature. \\
First of all, they fit into the research on invariance of stochastic processes with respect to random transformations (see, for example, \cite{Kallenberg2005} for Brownian motion, infinite dimensional Gaussian processes and $\alpha$-stable processes, or \cite{Privault2012} for multidimensional Poisson processes). In this setting, our approach allows us to consider general semimartingales with jumps and the use of It\^o stochastic calculus and semimartingales stochastic characteristics give us the possibility of working with explicit examples.\\
Furthermore, gauge symmetries represent a generalization of deterministic symmetries of stochastic processes defined on Lie groups or symmetric spaces (see \cite{Albeverio2007,Applebaum2015,Gangolli1964,Liao2016}). Indeed, in Proposition \ref{proposition_automorphism}, we prove that, if $\mathcal{G} \subset Aut(N)$ and $\Xi_g$ is the natural action of $Aut(N)$, then, for $G_t=g_0 \in \mathcal{G}$ deterministic, the natural transformation $\Xi_{g_0}(Z_t)$ and our random transformation $\Xi_{G_t}(dZ_t)=\Xi_{g_0}(dZ_t)$ coincide. This fact has an important consequence which generalizes the Brownian motion case, where the invariance with respect to deterministic rotations is equivalent to the invariance with respect to random rotations. Indeed as proven in Corollary \ref{corollary_Levy}, in the case where $\mathcal{G} \subset Aut(N)$, a L\'evy process, whose law is uniquely determined by its characteristics, is invariant with respect to our random transformations if and only if it is invariant with respect to deterministic transformations.\\

The natural field of application of our results is the study of the symmetries of general SDEs (see \cite{DMU2,DMU1,Glover1990,Cami2009,Liao1992,Liao2004,Liao2016} for the current literature and Section \ref{section_gauge} and \cite{DeVecchithesis} for some applications of our theory). The problem here is to find a diffeomorphism $\Phi:M \rightarrow M$ transforming a process $X$ that is a solution to a SDE into the process $\Phi(X)$ that solves the same SDE. If we look for weak solutions, the presence of gauge symmetries is fundamental in order to characterize the symmetries $\Phi$ of a given SDE. In fact,  as in the Brownian motion case, also in this framework, the law of the process $X$ is no more characterized by a single SDE, but it is determined by the whole family of SDEs related by a gauge transformation.\\
A second interesting application field is the geometry of stochastic processes on Riemannian manifolds. Indeed, the gauge invariance property can be useful for explaining the relationship between Brownian motion and Riemannian geometry (see, e.g.,  \cite{Elworthy1982,Emery1989}) or in the study of L\'evy processes taking values in Riemannian manifolds (see \cite{Applebaum2000}). See the end of Section \ref{section_gauge} for some ideas in this direction. \\

The paper is organized as follows. In Section \ref{section_SDE} we introduce the notion of geometrical SDEs and random transformations, and discuss their composition properties. In Section \ref{section_gauge} we define the concept of gauge symmetry providing some applications and in Section \ref{section_characteristics} we study the relationship between gauge symmetries and semimartingale characteristics, giving also some examples of gauge symmetric semimartingales.

\section{Random transformations and gauge symmetries of semimartingales}\label{section_SDE}

\subsection{Geometrical SDEs with jumps}\label{subsection_definition_SDE}

In this section we introduce a family of SDEs defined on a manifold $M$ and driven by a semimartingale taking values on a (finite dimensional) Lie group $N$. The family presented here is strongly inspired by the works of Cohen \cite{Cohen1996} (see also \cite{Applebaum1997,Cohen1995}), even if we introduce some minor changes in order to make these concepts more suitable to our theory. \\
Before starting with a detailed description of the notion of geometrical SDEs, let us consider a simple example. Given a discrete time process $Z$ defined on a Lie group $N$ and a measurable map $\Psi:M \times N \rightarrow M$, it is immediate to use the map $\Psi$ for defining, given the initial condition $X_0$, a new discrete time process $X$ on the manifold $M$ in the following way
\begin{equation}\label{equation_map}
 X_{n}=\Psi(X_{n-1},\Delta Z_n),
\end{equation}
where $\Delta Z_n:=Z_n \times (Z_{n-1})^{-1}$. In this setting the process $X$ can be seen as the solution process to a SDE defined by the map $\Psi$ and driven by the semimartingale $Z$. A geometrical SDE provides a generalization of the previous picture to the case of SDEs defined on a manifold and driven by a continuous time semimartingale. 
In the following we describe the notion of geometrical SDE, postponing to Section \ref{subsection_comparison} the comparison with other approaches which are more common in the current literature, as well as the the discussion of the relevance of this family of SDEs in our theory and of their wide applicability in stochastic analysis.\\

\noindent The principal object in our definition is a map
$$\Psi_{\cdot}(\cdot,\cdot):M \times N \times \mathcal{K} \rightarrow M,$$
where $\mathcal{K}$ is a topological space, 
 $\Psi_k$ is smooth in the $M,N$ variables, $\Psi_k(x, 1_N)=x$  and $\Psi_k$ and all its derivatives with respect to the $M,N$ variables are continuous in all their arguments. When we consider $\mathcal{K}$ consisting of a single point $\{k_0\}$, we write $\Psi$ instead of $\Psi_{k_0}$.\\
Let us define an auxiliary map $\overline{\Psi}_{\cdot}:M \times N \times N \times \mathcal{K} \rightarrow M$ defined as $\overline{\Psi}_k(x,z',z)=\Psi_k(x,z'\cdot z^{-1})$ and a process $K$ taking values on $\mathcal{K}$ which is \emph{predictable and locally bounded}, i.e. $K$ is a predictable process such that there exist an increasing sequence of stopping times $\tau_n \rightarrow + \infty$ and an increasing sequence of compact sets $\mathfrak{K}_n \subset \mathcal{K}$ such that $K_{t}(\omega) \in \mathfrak{K}_n$ whenever $0 < t \leq \tau(\omega)$.\\ 

\noindent In order to introduce the notion of solution to a SDE $\Psi_k$, we give the following definition of semimartingales on a manifold $M$ (see Emery \cite[Section 3.1]{Emery1989}).

\begin{definition}
An adapted c\acc{a}dl\acc{a}g stochastic process $X$ on a smooth manifold $M$ is a semimartingale if, for any smooth function $f \in
\cinf(M)$, the real-valued process $f(X)$ is a real-valued semimartingale.
\end{definition}

\noindent If $M$ and $N$ admit some global coordinate systems $x^i$ and $z^{\alpha}$ respectively, a semimartingale $X$ in $M$ is a solution to the SDE defined by $\Psi_{K_t}$ and driven by $Z$ if and only if, for a suitable stopping time $\tau$, we have
\begin{equation}\label{equation_manifold1}
\begin{array}{ccl}
X^i_{t\wedge \tau}-X^i_0&=&\int_0^{t\wedge \tau}{\partial_{z'^{\alpha}}(\overline{\Psi}^i_{K_s})(X_{s_-},Z_{s_-},Z_{s_-})dZ^{\alpha}_s}+\\
&&+\frac{1}{2}\int_0^{t\wedge \tau}{\partial_{z'^{\alpha}z'^{\beta}}(\overline{\Psi}^i_{K_s})(X_{s_-},Z_{s_-},Z_{s_-})d[Z^{\alpha},Z^{\beta}]^c_s}+\\
&&+\sum_{0\leq s \leq
{t\wedge \tau}}\{\overline{\Psi}^i_{K_s}(X_{s_-},Z_{s},Z_{s_-})-\overline{\Psi}^i_{K_s}(X_{s_-},Z_{s_-},Z_{s_-})+\\
&&-\partial_{z'^{\alpha}}(\overline{\Psi}^i_{K_s})(X_{s_-},Z_{s_-},Z_{s_-})\Delta
Z^{\alpha}_s\},
\end{array}
\end{equation}
where $\overline{\Psi}^i:=x^i \circ \overline{\Psi}$, $\partial_{z'^{\alpha}}$ denotes the derivative of $\overline{\Psi}^i(x,z',z)$ with respect to the second set $z'$ of variables on $N$ and with respect to the coordinates system $z^{\alpha}$, $X^i:=x^i(X)$, $Z^{\alpha}:=z^{\alpha}(Z)$, $\Delta Z^{\alpha}_s:=Z^{\alpha}_s-Z^{\alpha}_{s_-}$, and we use Einstein notation for repeated indexes.\\
In order to extend the previous definition to the case of two general smooth
manifolds $M,N$ we  introduce two embeddings $i_1:M
\rightarrow \mathbb{R}^{k_M}$ and $i_2:N \rightarrow
\mathbb{R}^{k_N}, k_M, k_N \in \mathbb{N}$, and an extension
$$\tilde{\Psi}_k: \mathbb{R}^{k_M} \times \mathbb{R}^{k_N} \times \mathbb{R}^{k_N} \times \mathcal{K} \rightarrow \mathbb{R}^{k_M},$$
of the map $\overline{\Psi}$ such that
$$\tilde{\Psi}_k(i_1(x),i_2(z'),i_2(z)))=i_1(\overline{\Psi}_k(x,z',z)).$$
The existence of such embeddings is given by the Whitney theorem (see, e.g., \cite[Chapter 1 Section 8]{Guillemin1974}) and the existence of such extension is guaranteed by the existence of a tubular neighbourhood for smooth submanifold (see, e.g., \cite[Chapter 2 Section 3]{Guillemin1974})). In the following with a slight abuse of notations we identify an SDE defined by $\overline{\Psi}$ or $\tilde{\Psi}$ with the SDE defined by $\Psi$.

\begin{definition}\label{definition_solution}
A semimartingale $X$ defined on the manifold $M$ solves the \emph{geometrical SDE} defined by $\Psi_{K_t}$ and driven by the noise $Z$ on the Lie group $N$ if, for any embedding $i_1,i_2$ as above, and for any extension $\tilde{\Psi}_k$ as above, $i_1(X_{t \wedge \tau}) \in \mathbb{R}^{k_M}$ solves the integral problem \refeqn{equation_manifold1}, where the map
$\overline{\Psi}_{K_t}$ is replaced by $\tilde{\Psi}_{K_t}$ and the noise $Z$ is replaced by $i_2(Z_{t \wedge \tau})$.\\
In this situation we write
$$dX_t=\Psi_{K_t}(dZ_t).$$
\end{definition}

\noindent When it is not strictly necessary, we omit the stopping time $\tau$ in the definition of solution to a SDE.

\begin{remark}\label{remark_map}
When the driving semimartingale $Z$ has predictable jumps at times $t=1,2,...,n,... \in \mathbb{N}$, Definition \ref{definition_solution} reduces to equation \refeqn{equation_map}.
\end{remark}

\begin{remark}\label{remark_definition}
In Definition \ref{definition_solution} the request that the map $\Psi_k$ is smooth in the $M \times N$ component can be relaxed. Indeed, it is only necessary to require that $\Psi_k$ is two times differentiable in the points of the form $(x,1_N)$ for any $x \in M$ and that the expression $\tilde{\Psi}(x,z',z)-i_1(x)-\partial_{z'^{\alpha}}(\tilde{\Psi})(x,z,z)\Delta z^{\alpha}$ is $O(|\Delta z|^2)$ uniformly on the compact sets of $M \times N$.
\end{remark}

\noindent Due to the following important theorem, Definition \ref{definition_solution} is in general not empty.

\begin{theorem}\label{theorem_manifold1}
Given two open subsets $M$ and $N$  of  $\mathbb{R}^m$ and
$\mathbb{R}^n$ respectively, for any semimartingale $Z$ on
$N$ and any $x_0 \in M$, there exist a stopping time $\tau$, almost surely strictly positive, and a semimartingale $X$ on $M$, uniquely defined until $\tau$ and such that $X_0=x_0$ almost surely, such that $X$ is a solution to the SDE $\Psi_{K_t}$ until the stopping time $\tau$ driven by $Z$.
Furthermore, if $M$ is a general (finite dimensional) manifold, $N$ a general (finite dimensional) Lie group and $Z$ is a semimartingale on $N$, $i_1,i_2$ are  two embeddings of $N,M$ in
$\mathbb{R}^{k_M}$ and $\mathbb{R}^{k_N}$ and $\tilde{\Psi}_k$ is any extension of $\overline{\Psi}_k$, then the unique solution
$\tilde{X}$ to the SDE $\tilde{\Psi}_k$ is of the form $(i_1(X),i_2(Z))$ for a unique semimartingale $X$ on $M$. Finally, the process
$X$ does not depend on the embeddings $i_1,i_2$ and on the extension $\tilde{\Psi}_k$.
\end{theorem}
\begin{proof}
Since the process $K$ is locally bounded, the function $\tilde{\Psi}_{K_t}$, up to a sequence of stopping times $\tau_n \rightarrow + \infty$,
is locally Lipschitz with Lipschitzianity constant uniform with respect to the point $\omega$ in the underlying probability space. Under these conditions, the statement follows from Theorem 2 in \cite{Cohen1996}.
${}\hfill$ \end{proof}

\subsection{Invariant geometrical SDEs on Lie groups}\label{subsection_invariant}

In this section we consider the special case $M=N$ (so we write $\Psi_k(\tilde{z},z) \in N$ with $\tilde{z} \in N$) and geometrical SDEs $\Psi_k$ invariant with respect to the group of right multiplication, i.e. $\Psi_k(\tilde{z} \cdot n, z)=\Psi_k(\tilde{z},z) \cdot n$ for any fixed $n \in N$. A simple consequence of these requests is that $\Psi_k$ is of the form
$$\Psi_k(\tilde{z},z)=\Psi^{\Xi}_k(\tilde{z},z):=\Xi_k(z) \cdot \tilde{z},$$
where $\Xi_k(z)=\Psi_k(1_N,z)$ is a smooth map from $N$ into itself such that $\Xi_k(1_N)=1_N$. In the following we write $d\tilde{Z}_t=\Xi_{K_t}(dZ_t)$ instead of $d\tilde{Z}_t=\Psi^{\Xi}_{K_t}(dZ_t)$.\\
In order to provide an heuristic reason for the introduction of geometrical SDE invariant with respect to the right multiplication on the Lie group $N$, we consider a discrete time semimartingale $Z$ (or equivalently a semimartingale $Z$ with predictable jumps at time $0,1,...,n,... \in \mathbb{N}$). In this case we have that $\tilde{Z}_n=\Xi_{K_{n-1}}(\Delta Z_n) \cdot \tilde Z_{n-1}$ or equivalently
$$\Delta \tilde{Z}_n=\Xi_{K_{n-1}}(\Delta Z_n)$$
(see Remark \ref{remark_map}). Therefore the map $\Xi_k$ describes the evolution of the jumps of the process $\tilde{Z}$ and, in particular, the jumps of the process $\tilde{Z}$ depend only on the jumps of the process $Z$ and not on the process $Z$ itself. \\

\noindent The invariant geometrical SDEs defined on the Lie group $N$ play the role of random transformations. Indeed, if we fix the locally bounded process $K$ in $\mathcal{K}$ and if $Z$ is a semimartingale on $N$, the unique solution to $d\tilde{Z}_t=\Xi_{K_t}(dZ_t)$ such that $\tilde{Z}_t=1_N$ is a new semimartingale on $N$. For this reason we call the map $Z \rightarrow \tilde Z$, where $\tilde Z$ is the solution to the SDE $d\tilde Z_t=\Xi_{K_t}(dZ_t)$, the \emph{random transformation of $Z$} related with the process $K$ and the map $\Xi_k$. Indeed a random transformation is a function from the set of semimartingales such that $Z_0=1_N$ into itself. Hereafter we assume that $Z_0=1_N$.\\

\noindent It is important to note that the random transformations just introduced could be seen (under some additional conditions) as an extension of deterministic transformations of a semimartingale. First of all, let $\mathcal{K}=Aut(N)$ be
the group of (smooth) automorphisms of $N$ (which is a finite dimensional Lie group) and let $\Xi_{\cdot}(\cdot):N \times \mathcal{K} \rightarrow N$ be the natural action of $Aut(N)$ on $N$, i.e. for any $k \in Aut(N)$ and $z, z' \in N$, $\Xi_{k}(z \cdot z')=\Xi_k(z) \cdot \Xi_k(z')$. Then the following proposition holds.

\begin{proposition}\label{proposition_automorphism}
In the previous setting for any (deterministic) $k \in Aut(N)$ and for any semimartingale $Z$, denoting by $\tilde{Z}$ the solution to $d\tilde{Z}_t=\Xi_k(dZ_t)$ and taking $Z'_t=\Xi_k(Z_t)$, we have $\tilde{Z}=Z'$.
\end{proposition}

\noindent Hereafter the equality between two stochastic processes means that the two processes are indistinguishable. Before proving Proposition \ref{proposition_automorphism} we fix some notations and introduce a lemma which we often use in the following. We denote by $Y_1,...,Y_n$ a set of generators for the Lie algebra of right-invariant vector fields on $N$, and we identify this Lie algebra with the Lie algebra $\mathfrak{n}$ associated with the Lie group $N$ (we recall that $\mathfrak{n}=T_{1_N}N$). Since $Y_{\alpha}$ are right invariant, denoting by $R_n:N \rightarrow N$ the right multiplication by the element $n \in N$, for any $f \in \cinf(N)$ we have that $Y_{\alpha}(f \circ R_n)(z)=Y_{\alpha}(f)(R_n(z))=Y_{\alpha}(f)(z \cdot n)$.\\
Let $S:N \rightarrow \mathbb{R}^{k_N}$ be an embedding of $N$ and let $P:N \rightarrow \matr(n,k_N)$ be the matrix defined as $P=(Y_{\alpha}(S^i))|_{\stackrel{\alpha=1,...,n}{i=1,...,k_N}}$. Since $S$ is an embedding (and so it is injective), the matrix $P$ is pointwise of maximal rank. This means that there exists a matrix $\tilde{P}:N \rightarrow \matr(k_N,n)$ defined as $\tilde{P}=(P^T \cdot P)^{-1} \cdot P^T$ which is pointwise the pseudoinverse of $P$, i.e. $\tilde{P} \cdot P=I_n$ and $P \cdot \tilde{P}|_{Im(P)}=Id|_{Im(P)}$ (see \cite{Penrose1955} for the definition of pseudoinverse).

\begin{lemma}\label{remark_extension}
Given a smooth function $f \in \cinf(N)$ and an embedding $S:N \rightarrow \mathbb{R}^{k_N}$, there exists an extension $\tilde{f} \in \cinf(\mathbb{R}^{k_N})$ of $f$ such that
\begin{eqnarray}
\partial_{s^i}(\tilde{f})\circ S &= & \tilde{P}^{\alpha}_i Y_{\alpha}(f) \label{equation_extension1}\\
\partial_{s^is^j}(\tilde{f}) \circ S&=&\tilde{P}^{\beta}_i \tilde{P}^{\alpha}_jY_{\beta}(Y_{\alpha}(\tilde{f})) \circ S +
Y_{\beta}(\tilde{P}^{\alpha}_j)\tilde{P}^{\beta}_iY_{\alpha}(\tilde{f}) \circ S \label{equation_extension2},
\end{eqnarray}
where $s^i$ is the coordinate system of $\mathbb{R}^{k_N}$.
\end{lemma}
\begin{proof}
The existence of a tubular neighbourhood $U_N \subset \mathbb{R}^{k_N}$ of the submanifold $S(N)$, and the existence of a partition of unity subordinate with respect to an atlas of $U_N$ adapted with respect to the submanifold $S(N)$ (see for example \cite[Chapter 1 Section 8]{Guillemin1974}), ensure that there exist an extension $\tilde{f} \in \cinf(\mathbb{R}^{k_N})$ of $f$ and a $k_N-n$ dimensional distribution $D \subset T_{S(N)}\mathbb{R}^{k_N}$ that is the orthogonal complement of $T(S(N))$ such that for any $Y,Y' \in D$ $Y(\tilde{f})|_{S(N)}=0$ and $Y(Y'(\tilde{f}))=0$.\\
Using the fact that any smooth vector field in $T_{S(N)}\mathbb{R}^{k_N}$ can be expressed as a linear combination of $S^*(Y_{\alpha})$ (the right-invariant vector fields of $N$) and of a vector field of $D$ and since $\tilde{P}$ is the pseudoinverse of $P$ and $D$ is the orthogonal complement of $TS(N)$, we get the statement.
\hfill\end{proof}\\

\noindent\begin{proof}[Proof of Proposition \ref{proposition_automorphism}]
The key tools of the proof are It\^o formula, Lemma \ref{remark_extension} and the right invariance of $Y_{\alpha}$. First of all we note that $\overline{\Psi}^{\Xi}_k=\Xi_k(z') \cdot \Xi_k(z^{-1}) \cdot \tilde{z}$.
Let $\tilde{\Psi}^{\Xi}_k$ be an extension of $\overline{\Psi}_k$ as in Definition \ref{definition_solution}  and Lemma \ref{remark_extension}, put $Z^i=S^i(Z)$ and let $\tilde{\Xi}_k$ be an extension of the type of Lemma \ref{remark_extension} of $S \circ \Xi_k$. By Lemma \ref{remark_extension} and the right-invariance of $Y_{\alpha}$ we have that
\begin{eqnarray*}
\partial_{s'^i}(\tilde{\Psi}^r_k)(\tilde{z},z',z)|_{S(N)}&=&\tilde{P}^{\alpha}_i Y^{z'}_{\alpha}(S^r \circ \overline{\Psi})(\tilde{z},z',z)\\
&=&\tilde{P}^{\alpha}_i Y_{\alpha}(S^r \circ \Xi_k)(z' \cdot z^{-1} \cdot \Xi^{-1}_k(\tilde{z}))\\
&=&\partial_{s^i}(\tilde{\Xi}^r_k)(z' \cdot z^{-1} \cdot \Xi^{-1}_k(\tilde{z}))|_{S(N)}.
\end{eqnarray*}
In the same way it is easy to prove, using equation \refeqn{equation_extension2}, that \\ $\partial_{s'^is'^j}(\tilde{\Psi}^r)(\tilde{z},z',z)|_{S(N)}=\partial_{s'^is'^j}(\tilde{\Xi}^r)(z' \cdot z^{-1} \cdot \Xi^{-1}_k(\tilde{z}))$.
By Definition  \ref{definition_solution} we have
\begin{eqnarray*}
d\tilde{Z}^i_t&=&\partial_{s^j}(\tilde{\Xi}^i_k)(\Xi^{-1}_k(\tilde{Z}_{t_-}))dZ^j_t+\partial_{s'^js'^r}(\tilde{\Xi}^i)(\Xi^{-1}_k(\tilde{Z}_{t_-})) d[Z^j,Z^r]_t+\\
&&+\tilde{\Xi}^i(\Delta Z_t \cdot \Xi^{-1}_k(\tilde{Z}_{t_-}))-\tilde{Z}^k_{t_-}-\partial_{s^j}(\tilde{\Xi}^i_k)(\Xi^{-1}_k(\tilde{Z}_{t_-}))\Delta Z^j_t.
\end{eqnarray*}
Using It\^o formula and the fact that $Z_t=\Xi^{-1}_k(Z'_t)$, we obtain that $Z'^i=S^i(Z')$ solves the previous differential relation and, by the uniqueness of the solution of geometrical SDEs, we have $\tilde{Z}=Z'$.
\hfill\end{proof}\\

\noindent The semimartingale $\tilde{Z}$, which is the random transformed semimartingale of $Z$ related to the process $K$ and the map $\Xi_k$, is the unique strong solution to the SDE $\Psi^{\Xi}_k$ driven by $Z$. In general the semimartingale $Z$ cannot be defined for all finite times $t$, since Theorem \ref{theorem_manifold1} assures only the local existence of $\tilde{Z}$. In order to guarantee that the map $Z \longmapsto \Xi_{K_t}(dZ_t)$ is a well defined map between semimartingales we need some conditions on $Z$, or $\Xi_k$ or $N$, ensuring that $\tilde{Z}$ has time of explosion $\tau=+\infty$.\\
For example, if $Z$ is piecewise constant with only a finite number of jumps in compact subsets of the time line $\mathbb{R}_+$, the semimartingale $\tilde{Z}$ is defined for all times. We give here an non-obvious statement that, under some conditions on the Lie group $N$, guarantees the existence of $\tilde{Z}$ for any time.

\begin{proposition}\label{proposition_faithful}
If $N$ admits a faithful finite dimensional representation, then, for any
locally bounded process $G_t$ in $\mathcal{G}$, the explosion time
of the SDE $d\tilde{Z}_t=\Xi_{G_t}(dZ_t)$ is $+
\infty$.
\end{proposition}
\begin{proof}
Let $S:N \rightarrow \matr(l_N,l_N)$ (where $l_N \in \mathbb{N}$) be a faithful representation
of $N$. In this representation, the geometrical SDE associated
with $\Xi_k$ is defined by the map $\overline{\Psi}^{\Xi}_k$ given by
$$\overline{\Psi}^{\Xi}_k(\tilde{z},z',z)=S(\Xi_k(z' \cdot z^{-1})) \cdot S(\tilde{z}),$$
where $\cdot$ on the right-hand side denotes the usual matrix multiplication. If $s^i$ is the standard cartesian coordinate system in
$\matr(l_N,l_N)$, extending suitably $\Xi_k$ to all $\matr(l_N,l_N)$, we have that $\overline{\Psi}^{\Xi,i}_k(\tilde{s},s',s)$,
$\partial_{s'^j}(\overline{\Psi}^{\Xi,i}_k)(\tilde{s},s',s)$ and $\partial_{s'^js'^l}(\overline{\Psi}^{\Xi,i}_k)(\tilde{s},s',s)$ are linear in $\tilde{s}$.
So, putting $Z^i=S^i(Z)$ and $\tilde{Z}^i=S^i(\tilde{Z})$, the SDE \refeqn{equation_manifold1} related with $\overline{\Psi}^{\Xi}$
is linear in $\tilde{Z}$ and so, by well known results on SDEs with jumps in $\mathbb{R}^{l_N^2}$ (see, e.g., \cite[Chapter 5]{Bichteler2002}) the
solution has explosion time $\tau=+ \infty$ almost surely. ${}\hfill$ \end{proof}\\

\noindent Hereafter we always assume that the semimartingale $\tilde{Z}$ exists for all times. 

\subsection{Geometrical SDEs and random transformations}

In this subsection we provide a natural composition rule for random transformations exploiting their relationship with geometrical SDEs. This composition rule is evident in the case where the semimartingale $Z$ has discrete time (or equivalently it has only jumps in the predictable times $t=0,1,...,n, n \in \mathbb{N}$). Indeed suppose that $d\tilde{Z}_n=\Xi_{K'_n}(dZ_n)$ and that $dX_n=\Psi_{K_{n}}(d\tilde{Z}_n)$, then by recalling Remark \ref{remark_map} and the discussion at the beginning of Section \ref{subsection_invariant} we have that
$$X_n=\Psi_{K_{n-1}}(X_{n-1},\Xi_{K'_{n-1}}(\Delta Z_n)).$$
In other words the semimartingale $X$ solution to the SDE $\Psi_k$ driven by $\tilde{Z}$ is also solution to the SDE $\Psi_k(x,\Xi_{k'}(z))$ driven by the semimartingale $Z$. \\
This fact it is not a peculiarity of the previous case but it is a general property of geometrical SDEs. Indeed the following theorem holds.

\begin{theorem}\label{theorem_gauge1}
Let $N$ be a Lie group and suppose that
$X$ is a solution to the geometrical SDE $\Psi_{K_t}$ driven by $\tilde{Z}$ on $N$. If
$d\tilde{Z}_t=\Xi_{K'_t}(dZ_t)$ (where $K'$ is a predictable locally bounded process defined in the topological space $\mathcal{K}'$),  then $X$ is a solution to the
geometrical SDE $\hat{\Psi}_{K_t,K'_t}$, where
$$\hat{\Psi}_{k,k'}(x,z)=\Psi_k(x,\Xi_{k'}(z)).$$
\end{theorem}

\noindent In order to prove Theorem \ref{theorem_gauge1} we need the following lemma.

\begin{lemma}\label{lemma_geometrical1}
Given $k$  c\acc{a}dl\acc{a}g semimartingales  $X^1,...,X^k$, let $H^{\alpha}_1,...,H^{\alpha}_k$, with $\alpha=1,...,r$, be predictable processes
which can be integrated along $X^1,...,X^k$ respectively. If $\Phi^{\alpha}(t,\omega,x^1,x'^1,...,x^k,x'^k):\mathbb{R}_+ \times \Omega
\times \mathbb{R}^{2k} \rightarrow \mathbb{R}$ are some progressively measurable random functions that are continuous in
$x^1,x'^1,...,x^k,x'^k$ and such that $|\Phi^{\alpha}(t,\omega,x^1,x'^1,...,x^k,x'^k)| \leq O((x^1-x'^1)^2+...+(x^k-x'^k)^2)$ as $x^i
\rightarrow x'^i$, for  almost every fixed $\omega \in \Omega$ and uniformly on compact subsets of $\mathbb{R}_+ \times \mathbb{R}^{2k}$, the processes
$$Z^{\alpha}_t=\int_0^t{H^{\alpha}_{i,s}dX^i_s}+\sum_{0 \leq s \leq t}\Phi^{\alpha}(s,\omega,X^1_{s_-},X^1_s,...,X^k_{s_-},X^k_s)$$
are semimartingales. Furthermore \normalsize
\begin{eqnarray}
&\Delta Z^{\alpha}_t=H^{\alpha}_{i,t} \Delta X^i_t + \Phi^{\alpha}(t,\omega,X^1_{t_-},X^1_t,...,X^k_{t_-},X^k_t),\label{equation_gauge4}\\
&[Z^{\alpha},Z^{\beta}]^c_t=\int_0^t{H^{\alpha}_{i,s}H^{\beta}_{j,s}d[X^i,X^j]^c_s},&\label{equation_gauge2}\\
&\int_0^t{K_{\alpha,s}dZ^{\alpha}_s}=\int_0^t{K_{\alpha,s}H^{\alpha}_{i,s}dX^i_s}+\sum_{0 \leq s \leq t}K_{\alpha,s}
\Phi^{\alpha}(s,\omega,X^1_{s_-},X^1_s,...).&\label{equation_gauge3}
\end{eqnarray}
\normalsize
\end{lemma}

\noindent \begin{proof}
Since $\int_0^t{H^{\alpha}_{i,s}dX^i_s}$ are semimartingales, we only need  to prove that\\ $\tilde{Z}^{\alpha} = \sum_{0 \leq s \leq t}\Phi^{\alpha}(s,\omega,X^1_{s_-},X^1_s,...,X^k_{s_-},X^k_s)$ is a c\acc{a}dl\acc{a}g process of bounded variation.\\
If $\tilde{Z}^{\alpha}$ are processes of bounded variation, then we can prove \refeqn{equation_gauge2}
since $\tilde{Z}^{\alpha}$ does not change the brackets $[Z^{\alpha},Z^{\beta}]^c$ (for the definition of $[~,~]^c$ see, e.\,g., \cite{Protter1990}). Furthermore, since $\tilde{Z}^{\alpha}$ is a sum of pure jumps processes, $\tilde{Z}^{\alpha}$  is a pure jump process. Then we get equations
\refeqn{equation_gauge4} and \refeqn{equation_gauge3} by using \cite[Chapter IV Theorem 8]{Protter1990}, the associative property of the It\^o integral (see \cite[Chapter IV Theorem 22]{Protter1990}), the fact that $\tilde{Z}^{\alpha}$ are pure jump processes of bounded variation and that the measures $d \tilde{Z}^{\alpha}$ are pure atomic (random) measures. The fact that $\tilde{Z}^{\alpha}$ is of bounded variation can be established by exploiting the standard argument used for proving It\^o formula (see, e.g., \cite[Chapter II, Section 7]{Protter1990}).
\hfill\end{proof}

\begin{remark}\label{remark_geometrical1}
Let $ \mathcal{K}$ be a topological space, $K \in \mathcal{K}$ be a locally bounded predictable process and $\tilde{\Phi}: \mathbb{R}_+ \times
\mathcal{K} \times \mathbb{R}^{2k} \rightarrow \mathbb{R}$ be a $C^2$ function in $\mathbb{R}^{2k}$ variables such that
$\tilde{\Phi}$ and all its derivatives are continuous in all their arguments. If
$\tilde{\Phi}(\cdot,\cdot,x^1,x^1,...,x^k,x^k)=\partial_{x'^i}(\tilde{\Phi})(\cdot,\cdot,x^1,x^1,...,x^k)=0$
 for $i=1,...k$, then $\Phi(t,\omega,...)=\tilde{\Phi}(t,K_t(\omega),...)$ satisfies the
hypotheses of Lemma \ref{lemma_geometrical1}.
\end{remark}

\noindent \begin{proof}[Proof of Theorem \ref{theorem_gauge1}]
We prove the theorem when $N,M$ admit a global coordinate system. The proof of the  general case can be obtained by using suitable embeddings and exploiting Lemma \ref{remark_extension}.\\
Let $x^i$, $z^{\alpha}$ 
be some global coordinate systems of $M,N$ respectively. By definition $\tilde{Z}$
is such that \normalsize
\begin{eqnarray*}
\tilde{Z}^{\alpha}_t-\tilde{Z}^{\alpha}_0&=&
\int_0^t{\partial_{z'^{\beta}}(\overline{\Psi}^{\Xi,\alpha}_{K'_s})(\tilde{Z}_{s_-},Z_{s_-},Z_{s_-})dZ^{\beta}_s}\\
&&+\frac{1}{2}\int_0^t{
\partial_{z'^{\beta}z'^{\gamma}}(\overline{\Psi}^{\Xi,\alpha}_{K'_s})(\tilde{Z}_{s_-},Z_{s_-},Z_{s_-})d[Z^{\beta},Z^{\gamma}]^c_s}+\\
&&+\sum_{0\leq s \leq t}\overline{\Psi}_{K'_s}^{\Xi,\alpha}(\tilde{Z}_{s_-},
Z_s,Z_{s_-})-\overline{\Psi}^{\Xi,\alpha}_{K'_s}(\tilde{Z}_{s_-},Z_{s_-},Z_{s_-})+\\
&&-\partial_{z'^{\beta}}(\overline{\Psi}^{\Xi,\alpha}_{K'_s})(\tilde{Z}_{s_-},Z_{s_-},Z_{s_-})\Delta Z^{\beta}_s,
\end{eqnarray*}
\normalsize
where, as usual, $\overline{\Psi}_{k'}(\tilde{z},z',z)=\Xi_{k'}(z'\cdot z^{-1}) \cdot \tilde{z}$. By the previous equation, Lemma
\ref{lemma_geometrical1} and Remark \ref{remark_geometrical1} we
obtain
\begin{eqnarray*}
[\tilde{Z}^{\alpha},\tilde{Z}^{\beta}]_t&=&\int_0^t{\partial_{z'^{\gamma}}
(\overline{\Psi}^{\Xi,\alpha}_{K'_s})(\tilde{Z}_{s_-},Z_{s_-},Z_{s_-})\partial_{z'^{\delta}}
(\overline{\Psi}^{\Xi,\beta}_{K'_s})(\tilde{Z}_{s_-},Z_{s_-},Z_{s_-})d[Z^{\gamma},Z^{\delta}]^c_s}\\
\Delta \tilde{Z}^{\alpha}_t&=&\overline{\Psi}^{\Xi,\alpha}_{K'_t}(\tilde{Z}_{t_-},
Z_t,Z_{t_-})-\overline{\Psi}^{\Xi,\alpha}_{K'_t}(\tilde{Z}_{t_-},Z_{t_-},Z_{t_-}).
\end{eqnarray*}
Therefore, since $X$ is a solution to the geometrical SDE $\Psi_{K_t}$ driven by $\tilde{Z}$, using Lemma
\ref{lemma_geometrical1} and Remark \ref{remark_geometrical1}, we have

$$\begin{array}{rcl}
dX^i_t&=&\partial_{\tilde{z}'^{\alpha}}(\overline{\Psi}^i_{K_t})(X_{t_-},\tilde{Z}_{t_-},\tilde{Z}_{t_-})d\tilde{Z}^{\alpha}_t+\frac{1}{2}
\partial_{\tilde{z}'^{\alpha}\tilde{z}'^{\beta}}(\overline{\Psi}^i_{K_t})(X_{t_-},\tilde{Z}_{t_-},\tilde{Z}_{t_-})\cdot \\

&&\cdot d[\tilde{Z}^{\alpha},\tilde{Z}^{\beta}]_t+\{\Psi^i_{K_t}(X_{t_-},\Delta \tilde{Z}_t)-\Psi^i_{K_t}(X_{t_-},1_N)+\\

&&-\partial_{\tilde{z}'^{\alpha}}(\overline{\Psi}^i_{K_t})(X_{t_-},\tilde{Z}_{t_-},\tilde{Z}_{t_-})\Delta \tilde{Z}^{\alpha}_t\}\\

\phantom{dX^i_t}&=&\partial_{\tilde{z}'^{\alpha}}(\overline{\Psi}^i_{K_t})(X_{t_-},\tilde{Z}_{t_-},\tilde{Z}_{t_-})\partial_{z'^{\beta}}
(\overline{\Psi}^{\Xi,\alpha}_{K'_t})
(\tilde{Z}_{t_-},Z_{t_-},Z_{t_-})dZ^{\beta}_t+\\

&&+\frac{1}{2}\partial_{\tilde{z}'^{\alpha}}(\overline{\Psi}^i_{K_t})(X_{t_-},\tilde{Z}_{t_-},\tilde{Z}_{t_-})
\partial_{z'^{\beta}z'^{\gamma}}(\overline{\Psi}^{\Xi,\alpha}_{K'_t})(\tilde{Z}_{t_-},Z_{t_-},Z_{t_-})\cdot\\
&&\cdot d[Z^{\beta},Z^{\gamma}]^c_t+\partial_{\tilde{z}'^{\alpha}}(\overline{\Psi}^i_{K_t})(X_{t_-},\tilde{Z}_{t_-},\tilde{Z}_{t_-})
\left[\overline{\Psi}_{K'_t}^{\Xi,\alpha}(\tilde{Z}_{t_-},Z_t,Z_{t_-})+\right.\\
&&\left.-\overline{\Psi}^{\Xi}_{K'_s}(\tilde{Z}_{t_-},Z_{t_-},Z_{t_-})-\partial_{z'^{\beta}}(\overline{\Psi}^{\Xi,\alpha}_{K'_t})(\tilde{Z}_{t_-},Z_{t_-},Z_{t_-})
\Delta Z^{\beta}_t\right]+\\
\phantom{dX^i_t}&\phantom{=}&+\frac{1}{2}\left(\partial_{\tilde{z}'^{\alpha}\tilde{z}'^{\delta}}(\overline{\Psi}^i_{K_t})(X_{t_-},\tilde{Z}_{t_-},\tilde{Z}_{t_-})
\partial_{z'^{\beta}}(\overline{\Psi}^{\Xi,\alpha}_{K'_t})(\tilde{Z}_{t_-},Z_{t_-},Z_{t_-})\cdot\right.\\

&&\left.\cdot\partial_{z'^{\beta}}(\overline{\Psi}^{\Xi,\delta}_{K'_t})(\tilde{Z}_{t_-},Z_{t_-},Z_{t_-})\right)d[Z^{\beta},Z^{\gamma}]_{t}+\left\{
\Psi^i(X_{t_-},\Xi_{G_t}(\Delta Z_t))+\right.\\
&&-\Psi^i(X_{t_-},1_N)-\partial_{\tilde{z}'^{\alpha}}(\overline{\Psi}^i_{K_t})(X_{t_-},\tilde{Z}_{t_-},\tilde{Z}_{t_-})[\overline{\Psi}_{K'_t}^{\Xi,\alpha}
(\tilde{Z}_{t_-}, Z_s,Z_{t_-})+\\
&&\left.-\overline{\Psi}_{K'_s}^{\Xi,\alpha}(\tilde{Z}_{t_-},Z_{t_-},Z_{t_-})]\right\}.
\end{array}
$$
\normalsize By the chain rule for derivatives and the fact that
$\overline{\Psi}^{\Xi,\alpha}_{K'_t}(\tilde{Z}_{s_-},Z_{s_-},Z_{s_-})=\tilde{Z}^{\alpha}_{s_-}$ we have 
$$\begin{array}{c}
\left.\partial_{z'^{\beta}}(\overline{\Psi}^i_{K_s}(x,\overline{\Psi}^{\Xi}_{K'_s}(\tilde{z},z',z),\tilde{z})) \right|_{\stackrel{
x=X_{s_-},\tilde{z}=\tilde{Z}_{s_-}}{z=z'=Z_{s_-}}}=\\
=
\partial_{\tilde{z}'^{\alpha}}(\overline{\Psi}^i_{K_s})(X_{s_-},\tilde{Z}_{s_-},\tilde{Z}_{s_-})\partial_{z'^{\beta}}(\overline{\Psi}^{\Xi,\alpha}_{K'_s})
(\tilde{Z}_{s_-},Z_{s_-},Z_{s_-})\\
\left.\partial_{z'^{\beta}z'^{\gamma}}(\overline{\Psi}^i_{K_s}(x,\overline{\Psi}^{\Xi}_{K'_s}(\tilde{z},z',z),\tilde{z})) \right|_{\stackrel{
x=X_{s_-},\tilde{z}=\tilde{Z}_{s_-}}{z=z'=Z_{s_-}}}=\\
=\partial_{\tilde{z}'^{\alpha}}(\overline{\Psi}^i_{K_s})(X_{s_-},\tilde{Z}_{s_-},\tilde{Z}_{s_-})
\partial_{z'^{\beta}z'^{\gamma}}(\overline{\Psi}^{\Xi,\alpha}_{K'_s})(\tilde{Z}_{s_-},Z_{s_-},Z_{s_-})+\\
+\partial_{\tilde{z}'^{\alpha}\tilde{z}'^{\delta}}(\overline{\Psi}^i_{K_s})(X_{s_-},\tilde{Z}_{s_-},\tilde{Z}_{s_-})
\partial_{z'^{\beta}}(\overline{\Psi}^{\Xi,\alpha}_{K'_s})(\tilde{Z}_{s_-},Z_{s_-},Z_{s_-})\cdot\\
\cdot\partial_{z'^{\gamma}}(\overline{\Psi}^{\Xi,\delta}_{K'_s})(\tilde{Z}_{s_-},Z_{s_-},Z_{s_-})
\end{array}
$$
\normalsize
Using the fact that
$$\begin{array}{c}
\overline{\Psi}_k(x,\overline{\Psi}^{\Xi}_{k'}(\tilde{z},z',z),\tilde{z})=\Psi_k(x,(\Xi_{k'}(z' \cdot z^{-1})\cdot \tilde{z})\cdot \tilde{z}^{-1})=\\
=\hat{\Psi}_{k,k'}(x,z' \cdot z^{-1})=\Psi_k(x,\Xi_{k'}(z'\cdot z^{-1}))= \overline{\hat{\Psi}}_{k,k'}(x,z',z) \end{array}$$
we obtain
\normalsize
\begin{eqnarray*}
X^i_t-X^i_0&=&\int_0^t{\partial_{z'^{\beta}}(\overline{\hat{\Psi}}^i_{K_s,K'_s})(X_{s_-},Z_{s_-},Z_{s_-})dZ^{\beta}_s}+\\
&&+\frac{1}{2}\int_0^t{
\partial_{z'^{\beta}z'^{\gamma}}(\overline{\hat{\Psi}}^i_{K_s,K'_s})(X_{s_-},Z_{s_-},Z_{s_-})d[Z^{\beta},Z^{\gamma}]_s}+\\
&&+\sum_{0\leq s \leq t}\overline{\hat{\Psi}}_{K_s,K'_s}^i(X_{s_-},Z_{s},Z_{s_-})-\hat{\Psi}_{K_s,K'_s}^i(X_{s_-},Z_{s_-},Z_{s_-})+\\
&&-
\partial_{z'^{\beta}}(\overline{\hat{\Psi}}^i_{K_s,K'_s})(X_{s_-},Z_{s_-},Z_{s_-})\Delta Z^{\beta}_s,
\end{eqnarray*}
\normalsize and so $dX_t=\hat{\Psi}_{K_t,K'_t}(dZ_t)$. ${}\hfill$ \end{proof}

\begin{corollary}\label{corollary_gauge1}
Suppose $\Xi_{k'}$ is invertible as a map from $N$ into itself and, for any fixed $k'$ in $\mathcal{K}'$, the inverse satisfies the same regularity properties of $\Xi_{k'}$. If $X$ is a solution to the geometrical SDE $\Psi_{K_t}$ driven by $Z$, then $X$ is a solution to
the canonical SDE defined by
$\hat{\Psi}_{k,k'}(x,z)=\Psi_k(x,\Xi^{-1}_{k'}(z))$.
\end{corollary}
\begin{proof}
The proof is an application of Theorem \ref{theorem_gauge1} and of the fact that $dZ_t=\Xi^{-1}_{K'_t}(d\tilde{Z}_t)$. Indeed, defining
$d\hat{Z}_t=\Xi^{-1}_{K'^{-1}}(d\tilde{Z}_t)$, we have that $d\hat{Z}_t=\Xi^{-1}_{K'_t} \circ \Xi_{K'_t}
(dZ_t)=Id_N(dZ_t)=dZ_t$.
${}\hfill$ \end{proof}

\begin{remark}An easy consequence of Theorem \ref{theorem_gauge1} is the well known associative rule of It\^o stochastic integrals. In particular if we have
$$dX^i_t=C^i_{\alpha,t} d\tilde{Z}^{\alpha}_t, \ \ \ d\tilde{Z}^{\alpha}_t=B^{\alpha}_{\beta,t} dZ^{\beta},$$
where $C:\Omega \times \mathbb{R}_+ \rightarrow \matr(m,n)$ and $B: \Omega \times \mathbb{R}_+ \rightarrow \matr(n,n)$ are two predictable locally bounded processes, we have
$$dX^i_t=C^i_{\alpha,t}B^{\alpha}_{\beta,t} dZ^{\beta}=D^i_{\beta,t} dZ^{\beta}_t$$
where $D_t=C_t \cdot B_t$. The previous relation can be obtained considering $X$ as the solution to the geometrical SDE given by $\Psi_C(x,z)=x+C \cdot z$ and by the random transformation $\Xi_B=B \cdot z$. Indeed applying Theorem \ref{theorem_gauge1} we obtain that $\hat{\Psi}_{C,B}(x,z)=x+C \cdot B \cdot z$.
\end{remark}

\subsection{A comparison with other approaches}\label{subsection_comparison}

In this subsection we discuss the relations between geometrical SDEs introduced by Definition \ref{definition_solution} and other definitions of SDEs driven by processes with jumps on $\mathbb{R}^m$. In particular we consider Marcus-type SDEs (see for example in \cite{Protter1995,Marcus1981}) and L\'evy driven SDEs (see for example \cite{Applebaum2004,Kunita2004}). Our geometrical SDEs also include affine-type SDEs (see \cite[Chapter V]{Protter1990}), SDEs driven by a general random measure (see \cite[Chapter 5]{Bichteler2002}) and smooth iterated random functions (see \cite{Diaconis1999}).

\subsubsection{Marcus-type SDEs}

Let us take $M= \mathbb{R}^m$ and $N=\mathbb{R}^n$, and $n$ smooth vector fields $V_1,...,V_n$ on $M$, admitting flow for any time. Let $R: [0,1] \times N \rightarrow N$ be $C^1$ in the $[0,1]$-variable, smooth in $N$-variables and such that $R(0,z)=0$, $R(1,z)=z$ and $R(a,0)=0$. If we define $\Phi:[0,1] \times M \times N \rightarrow M$ as the unique solution to the following equations
$$\begin{array}{rcl}
\partial_a(\Phi^i(a,x,z))&=&\partial_a(R^{\alpha}(a,z)) V_{\alpha}^i(\Phi(a,x,z))\\
\Phi(0,x,z)&=&x,
\end{array}
$$
where $V_{\alpha}=V^i_{\alpha}(x)\partial_{x^i}$, and we consider $\Psi(x,z)=\Phi(1,x,z)$, the geometrical SDE defined by $\Psi$ is given by
\begin{equation}\label{equation_marcus}
\begin{array}{rcl}
dX^i_t&=&V^i_{\alpha}(X_{t_-})dZ^{\alpha}_t+\frac{1}{4}(V_{\beta}(V^i_{\alpha})(X_{t_-})
+V_{\alpha}(V^i_{\beta})(X_{t_-}))d[Z^{\alpha},Z^{\beta}]_t+\\
&&+\{\Psi^i(X_{t_-},Z_t-Z_{t_-})-X^i_{t_-}-V^i_{\alpha}(X_{t_-})\Delta Z^{\alpha}_t\}.
\end{array}
\end{equation}
We call these SDEs \emph{Marcus-type SDEs} (some authors call them \emph{canonical SDEs}) since when $R(a,z)=a z$,  they are exactly the SDEs initially proposed by Marcus in \cite{Marcus1981} for semimartingales with finitely many jumps in any compact interval, which have been extended to the case of general real semimartingales in \cite{Protter1995}. This kind of equations are used to study stochastic flows on $\mathbb{R}^m$ (see \cite{Kunita1999}) and can be defined on manifolds and driven by infinite dimensional semimartingales (see \cite{Applebaum2001}). Furthermore, when $R$ is a more general class, they have been studied by Cohen (see \cite{Cohen1996} and see also \cite{Applebaum1997}) within the framework of stochastic analysis and by Friz and Zhang (see \cite{Friz2017} and also \cite{Friz2017(2)}) exploiting the more recent methods of rough paths theory. The Marcus-type SDEs have a very nice geometric property: if $\Sigma:M \rightarrow M$ is a diffeomorphism, $X$ solves the SDE \refeqn{equation_marcus} if and only if $\Sigma(X)$ solves the Marcus-type SDE defined by the same $R$ and by $\Sigma^*(V_1),...,\Sigma^*(V_n)$. Unfortunately this family of SDEs has two main problems. The first one is that they do not have a natural generalization allowing us to include the dependence on topological spaces $\mathcal{K}$ and stochastic processes on them. The second problem is that, although they are closed with respect to diffeomorphisms, they are not closed with respect to the random transformations $\Xi_k$. \\
Nevertheless since some generalizations of equation \refeqn{equation_marcus} play an important role in (possible non-It\^o) rough paths theory for jumps processes, we think that a generalization of our theory to these cases and to rough paths theory deserves a great attention and future investigations.

\subsubsection{Smooth SDEs driven by L\'evy processes}\label{subsubsection_Levy}

In this section we describe a particular form of SDEs driven by $\mathbb{R}^n$-valued L\'evy processes (see, e.g.,
\cite{Applebaum2004,Kunita2004}). The following discussion can be easily generalized to the case of general semimartingales and general random measures following \cite{Bichteler2002}. By definition, an $\mathbb{R}^n$-valued L\'evy process $(Z^1,...,Z^n)$ can be decomposed into the sum of Brownian
motions and compensated Poisson processes defined on $\mathbb{R}^n$. In particular, a L\'evy process on $\mathbb{R}^n$ can be identified by a
vector $b_0=(b_0^1,...,b_0^n) \in \mathbb{R}^n$, an $n \times n$ matrix $A_0^{\alpha\beta}$ (with real elements) and a positive $\sigma$-finite measure $\nu_0$
defined on $\mathbb{R}^n$ (called L\'evy measure, see, e.g., \cite{Applebaum2004,Sato1999}) such that
$$\int_{\mathbb{R}^n}{\frac{|z|^2}{1+|z|^2}\nu_0(dz)} < + \infty.$$
By  the L\'evy-It\^o decomposition, the triplet $(b,A,\nu)$ is such that there exist an $n$ dimensional Brownian motion $(W^1,...,W^n)$ and a
Poisson measure $P(dz,dt)$ defined on $\mathbb{R}^n$ such that \normalsize
$$
\begin{array}{rcl}
Z^{\alpha}_t&=&b^{\alpha}_0 t+ C^{\alpha}_{\beta} W^{\beta}_t+\int_0^t{\int_{|z| \leq
1}{z^{\alpha}(P(dz,ds)-\nu_0(dz)ds)}}+\\
&&+\int_0^t{\int_{|z|>1}{z^{\alpha}P(dz,ds)}}. \end{array}$$ \normalsize where
$A^{\alpha\beta}_0=\sum_{\gamma}C^{\alpha}_{\gamma}C^{\beta}_{\gamma}$. Henceforth we suppose for simplicity that $b^1=1$ and $b^{\alpha}_0=0$
for $\alpha>1$, that there exists
 $n_1$ such that $A^{\alpha\beta}_0=\delta^{\alpha\beta}$ for $1< \alpha ,\beta \leq n_1$ and
 $A^{\alpha\beta}_0=0$ for $\alpha$ or $\beta$ in $\{1,n_1+1,...,n\}$, and finally  that
 $\int_0^t{\int_{|z| \leq 1}{z^{\alpha}(P(dz,ds)-\nu_0(dz)ds)}}=0$ and $\int_0^t{\int_{|z|>1}{z^{\alpha}P(dz,ds)}}=0$ for $\alpha \leq n_1$. \\
Consider a vector field $\mu$ on $M$, a set of $n_1-1$ vector
fields $\sigma=(\sigma_2,...,\sigma_{n_{1}})$ on $M$ and a smooth
(both in $x $ and $z$) function $F:M \times \mathbb{R}^{n-n_1}
\rightarrow \mathbb{R}^m$ such that $F(x,0)=0$. We say that a
semimartingale $X \in M$ is a solution to the smooth SDE
$(\mu,\sigma,F)$ driven by the $\mathbb{R}^n$ L\'evy process
$(Z^1,...,Z^n)$ if
\begin{eqnarray*}
X^i_t-X^i_0&=&\int_0^t{\mu^i(X_{s_-})dZ^1_s}+\int_0^t{\sum_{\alpha=2}^{n_1}\sigma^i_{\alpha}(X_{s_-})dZ_s^{\alpha}}+\\
&&+\int_0^t{\int_{\mathbb{R}^{n-n_1}}{F^i(X_{s_-},z)(P(dz,ds)-I_{|z|\leq 1}\nu_0(dz)ds)}},
\end{eqnarray*}
where $I_{|z|\leq 1}$ is the indicator function of the set
$\{|z|\leq 1\}\subset \mathbb{R}^{n-n_1}$. Define the function
$$\overline{\Psi}^i(x,z',z)=x^i+\tilde{\mu}^i(x)(z'^1-z^1)+\sigma^i_{\alpha}(x)(z'^{\alpha}-z^{\alpha})+F^i(x,z'-z),$$
where
$$\tilde{\mu}^i(x)=\mu^i(x)-\int_{|z|\leq 1}{(F^i(x,z)-\partial_{z^{\alpha}}(F^i)(x,z)z^{\alpha})\nu_0(dz)}.$$
It is easy to see that any solution $X$ to the smooth SDE
$(\mu,\sigma,F)$ driven by the L\'evy process $(Z^1,...,Z^n)$ is
also a solution to the geometrical SDE $\overline{\Psi}$ driven by
the $\mathbb{R}^n$ semimartingale $(Z^1,...,Z^n)$ and conversely.\\
In the theory of SDEs driven by $\mathbb{R}^n$-valued L\'evy processes the usual assumption  is
that $F$ is Lipschitz in $x$ and measurable in $z$. Our assumption on smoothness of $F$ in both $x,z$ is thus  a stronger requirement. For this reason
we say that $(\mu,\sigma,F)$ is a \emph{smooth SDE} driven by a L\'evy process.

\section{Gauge symmetries of semimartingales on Lie groups}\label{section_gauge}

It is well known that considering  a Brownian
motion $Z$ on $\mathbb{R}^n$ and a process $B_t: \Omega \times [0,T]
\rightarrow O(n)$  predictable  with respect to the natural filtration of $Z$ and with
 values in the Lie group $ O(n)$ of orthogonal matrices,  the process defined by
\begin{equation}\label{equation_gauge1}
Z'^{\alpha}_t=\int_0^t{B^{\alpha}_{\beta,s}dZ^{\beta}_s}
\end{equation}
is a new $n$ dimensional Brownian motion (see \cite{DMU1}).\\
We propose a generalization of this property to the case of a c\acc{a}dl\acc{a}g semimartingale $Z$  on a Lie group $N$ (see
\cite{Privault2012} for a similar result about Poisson measures).\\

\noindent In the following we suppose that $\Xi:N \times \mathcal{G} \rightarrow N$ is an action of the topological group $\mathcal{G}$, i.e. $\Xi_g$ satisfies all the previous hypotheses where the topological space $\mathcal{K}$ is replaced by the topological group $\mathcal{G}$ and furthermore, for any $g_1,g_2 \in \mathcal{G}$, $\Xi_{g_1} \circ \Xi_{g_2} =\Xi_{g_1 \cdot g_2}$.

\begin{definition}\label{definition_gauge}
Let $Z$ be a semimartingale on a Lie group $N$ with respect to a given filtration $\mathcal{F}_t$. Given a topological group $\mathcal{G}$, 
 we say that $Z$ admits $\mathcal{G}$, with action $\Xi_g$ and with respect to the filtration $\mathcal{F}_t$,  as \emph{gauge symmetry group} if, for any $\mathcal{F}_t$-predictable locally bounded process $G_t$ taking values in $\mathcal{G}$, the semimartingale $\tilde{Z}$ solution to the equation $d\tilde{Z}_t=\Xi_{G_t}(dZ_t)$ has the same law as $Z$.
\end{definition}

\noindent In the following we consider the filtration $\mathcal{F}_t$ of the  probability space $(\Omega,\mathcal{F},\mathbb{P})$ as given and we omit it when it is not strictly necessary to specify it.

\begin{remark}
We require that $\mathcal{G}$ is a topological group and $\Xi_g$ is a group action, so that the set of random maps of the form $Z \longmapsto \Xi_{G_t}(dZ)$, for a generic predictable locally bounded process $G$, forms a group of random transformations. Under the previous conditions this group property is an easy consequence of Theorem \ref{theorem_gauge1} and Corollary \ref{corollary_gauge1}. On the other hand, we can modify Definition \ref{definition_gauge} requiring a weaker condition, i.e. $\mathcal{G}=\mathcal{K}$ to be a general topological space or a semigroup. The assumption that $\mathcal{G}$ is a group is essentially based on the idea that symmetries should be invertible and closed under composition.
\end{remark}

\noindent There are two main reasons for the choice of the name \emph{gauge symmetry group}. If we consider the space $\mathfrak{G}$ of locally bounded predictable functions $G_{\cdot}:\Omega \times \mathbb{R}_+ \rightarrow \mathcal{G}$, $\mathfrak{G}$ is a group with respect to the pointwise composition $(G \cdot G')_t(\omega)=G_t(\omega) \cdot G'_t(\omega)$. The set $\mathfrak{G}$ with this composition can be interpreted as the set of gauge transformations of the trivial principal bundle $(\Omega \times \mathbb{R}_+) \times \mathcal{G}$. Furthermore, by Theorem \ref{theorem_gauge1}, the map $Z_t \times G_t \rightarrow \Xi_{G_t}(dZ_t)$ is an action of the group $\mathfrak{G}$ on the space of semimartingales defined on $N$. This resembles the idea of (local) gauge transformations and (local) gauge symmetries in field theory.\\
A deeper similarity between field theory and gauge symmetries in relation with Markovian (in $X$) SDEs is based on the idea of weak solution to a SDE. If we fix the law $\mathbb{P}^Z$ of a semimartingale $Z$ on the Lie group $N$, we say that $X$ is \emph{a weak solution to the geometrical SDE $\Psi$} if there exists a semimartingale $Z'$ with the law $\mathbb{P}^Z$ taking values on $N$ such that $X$ is a solution to the SDE $\Psi$ driven by $Z'$. This definition, inspired by \cite[Definition 5.5.2]{Bichteler2002}, is a natural generalization of the case with a driving Brownian-motion, and can be extended from geometrical SDEs to more general situations.\\

\begin{remark}
The strong existence results proved in Theorem \ref{theorem_manifold1} and the coincidence between strong and weak solutions allow us to choose $Z'$ so that $X$ is measurable with respect to the natural filtration of $Z'$. \\
Obviously this is no more true for the general non-smooth case (see Tanaka counter-example  \cite[Chapter 5, Section 5.5]{Bichteler2002}).
\end{remark}

\noindent Let $\mathcal{G}$ be a (finite dimensional) gauge symmetry Lie group for the semimartingale $Z$ with smooth action $\Xi_g$ with respect to its natural filtration. If $B:M \rightarrow \mathcal{G}$ is a smooth function and  $X$ is a weak solution to the geometrical SDE $\Psi(x,z)$, then $X$ is also a weak solution to the SDE $\Psi(x,\Xi_{B(x)}(z))$. Therefore there is a whole family of SDEs $\Psi$ with the same set of weak solutions which are related by transformations of the form $\Xi_{B(x)}$. \\
This fact is well known in the Brownian case, where with the same martingale problem it is possible to associate many different SDEs related by a rotation of the diffusion coefficient. With the introduction of the concept of gauge symmetry we generalize this phenomenon to gauge symmetric semimartingales. \\
The situation is similar to the one encountered in classical electrodynamics, where with the same electro-magnetic field it is possible to associate many different electro-magnetic scalar and vector potentials related by a gauge transformation. Furthermore this approach allows us to relate the gauge symmetry group with the (local) gauge transformation on the trivial bundle $M \times \mathcal{G}$ (see \cite{DMU1} for the Brownian motion case). \\

\noindent The existence of gauge symmetries is important in the study of the \emph{symmetries of SDEs}, i.e. transformations changing a solution to an SDE into another solution to the same SDE. Since we have two notions of solution to a given SDE (strong and weak solution) we have two notions of symmetries for an SDE: \emph{strong symmetries}, leaving invariant the set of strong solutions (i.e. not changing the driving process) and \emph{weak symmetries}, leaving invariant the set of weak solutions. Characterizing the set of strong symmetries is a simple task since the SDE must be transformed into itself. On the other hand, when the driving semimartingale admits a gauge symmetry group, there are many different SDEs admitting the same set of weak solutions, and a weak symmetry may not transform the SDE into itself.\\
The relationship between strong and weak symmetries of a Brownian-motion driven SDE and the gauge symmetries with respect to random rotations of Brownian motion was pointed out for the first time in \cite{DMU1}. The study of weak symmetries of SDEs driven by general semimartingales admitting gauge symmetry groups will be the subject of a future paper.\\

\noindent The presence of gauge symmetries of semimartingales has also important consequences for the study of the geometry of stochastic processes defined on manifolds. We sketch here an example. A well known method for introducing Brownian motion on a Riemannian manifold $M$ is to project onto the manifold $M$ the solution to a SDE defined on the bundle of orthonormal frames $O(M)$ on $M$ (the solution to the SDE on $O(M)$, called stochastic development of the Brownian motion on $M$, was first proposed by Eells and Elworthy, see \cite{Elworthy1982}, see also the generalization of Emery to general continuous semimartingales provided in \cite[Chapter VIII]{Emery1989}). This construction was generalized to the case of rotation-invariant L\'evy processes in \cite{Applebaum2000}. It is important to note the central role of rotation invariance for Brownian motion and for the L\'evy processes in proving that the projected process from $O(M)$ onto $M$ is Markovian (indeed in \cite{Applebaum2000} it is explicitly acknowledge that for non-rotation-invariant L\'evy processes this fact is no necessarily true anymore). This phenomenon, once explained in terms of gauge symmetries, can be generalized to the case of non-Markovian-driving processes.\\
Indeed, the stochastic development defined on $O(M)$ can be related to a Marcus-type SDE $\Psi(x,z)$ (where $x \in O(M)$ and $z \in N=\mathbb{R}^m$). The map $\Psi$ cannot be reduced to a map $\hat{\Psi}:M \times N \rightarrow M$, i.e. a map $\pi(\Psi(x,z))=\hat{\Psi}(\pi(x),z)$ where $\pi:O(M) \rightarrow M$ is the natural projection of $O(M)$ on $M$. This means that the SDE $\Psi$ driven by $Z$ on $O(M)$ cannot be reduced to a SDE $\hat{\Psi}$ driven by $Z$ on $M$. On the other hand the existence of this kind of reduction is fundamental, since the presence of $\hat{\Psi}$ would imply the Markovianity of the projection $\Pi(X)$ of the process $X$ on $M$. If $Z \in \mathbb{R}^m$ is invariant with respect to random rotations $\Xi_B(z)=B \cdot z$ (where $B \in O(m)$), the choice of an orthonormal frame on $M$\footnote{The existence of a smooth orthonormal frame is equivalent to the triviality of the tangent bundle $TM$. If $TM$ is non-trivial, then there exist only measurable orthonormal frames which can be used if we apply the generalization of geometrical SDEs proposed in Remark \ref{remark_definition}.} induces a map $B:O(M) \rightarrow O(m)$ such that the SDE $\Psi(x,\Xi_{B(x)}(z))=\Psi(x,B(x) \cdot z)$ can be reduced to a SDE $\hat{\Psi}(\tilde{x},z)$ on $M$. Since, as proved in Corollary \ref{corollary_Levy}, rotation-invariant L\'evy processes are also invariant with respect to random rotations, the previous reasoning also applies to Brownian motion and the L\'evy processes considered in \cite{Applebaum2000}. \\

\noindent We conclude this section with an explicit example showing how the knowledge of a gauge symmetry of a semimartingale provides a useful tool in order to simplify the corresponding SDE. \\
Given  the  SDE $\Psi(x,z)=x+xf(|x|^2)z_0+z$ in $\mathbb{R}^2$ or, more explicitly,
\begin{eqnarray*}
dX^1_t&=&X^1_{t_-}f(R_{t_-})dt+dZ^1_{t_-}\\
dX^2_t&=&X^2_{t_-}f(R_{t_-})dt+dZ^2_{t_-}
\end{eqnarray*}
where $f:\mathbb{R}_+ \rightarrow \mathbb{R}$ is a smooth function and $R_t=(X^1_t)^2+(X^2_t)^2$, it
is easy to prove that (for general $f$) $\Psi(x,z)$  does not admit any strong symmetry $\Phi$. Indeed, in general, we cannot find a smooth function $\Phi:\mathbb{R}^2\rightarrow \mathbb{R}^2$ such that $\Phi(\Psi(\Phi^{-1}(x),z))=\Psi(x,z)$ for any $(x,z)\in\mathbb{R}^4$. \\
On the other hand, if $Z$ admits $O(2)$ (with its natural action) as gauge symmetry
group, $\Psi$ admits the rotation
$$\Phi_a(x)=\left( \begin{array}{cc} \cos(a) & -\sin(a)\\
\sin(a) & \cos(a)
\end{array} \right) \cdot \left( \begin{array}{c} x^1\\
x^2
\end{array} \right)
 $$ as weak symmetries. Indeed, if
$$B_a=\left( \begin{array}{cc} \cos(a) & -\sin(a)\\
\sin(a) & \cos(a)
\end{array} \right), $$
 we have that $(\Phi_a,B_a)$ are weak symmetries (see \cite{DMU1} for a precise definition) and
\begin{eqnarray*}
dX'_t&=&d\Phi_a(X)=B_a \cdot dX_t\\
&=&f(R_{t_-}) B_a\cdot  X_{t_-} dt + B_a \cdot  dZ_t\\
&=&f(R'_{t_-})X'_{t_-} dt+ dZ'_{t_-}.
\end{eqnarray*}
We remark that, in order to have  the previous symmetry, we do not need $Z$ to be invariant with respect to random rotations but only with respect to deterministic rotations.
On the other hand, the notion of gauge invariance is a key tool if we aim at generalizing the reduction and reconstruction  techniques from the deterministic to  the stochastic framework. \\
Indeed, rewriting the equation in polar coordinates, we obtain
\begin{equation}\label{equation_radius}
dR_t=2R_{t_-}f(R_{t_-})dt+2\sqrt{R_{t_{-}}} (\cos(\Theta_{t_-})dZ^1_{t}+\sin(\Theta_{t_-})dZ^2_t)+ d\mathfrak{Z}_t,
\end{equation}
where $\mathfrak{Z}_t=[Z^1,Z^1]_t+[Z^2,Z^2]_t$. It is immediate to check that equation \eqref{equation_radius} is not independent of the angle $\Theta_t$ as expected for a rotationally invariant equation. On the other hand, if $Z$ admits rotations as gauge symmetry group, we can define the following semimartingale
\begin{eqnarray*}
Z'^1_t&=&\int_0^t{ \frac{X^1_{s_-}}{\sqrt{(X^1_{s_-})^2+(X^2_{s_-})^2}} dZ^1_s + \frac{X^2_{s_-}}{\sqrt{(X^1_{s_-})^2+(X^2_{s_-})^2}} dZ^2_s}\\
Z'^2_{t}&=&\int_0^t{-\frac{X^2_{s_-}}{\sqrt{(X^1_{s_-})^2+(X^2_{s_-})^2}} dZ^1_s+
\frac{X^1_{s_-}}{\sqrt{(X^1_{s_-})^2+(X^2_{s_-})^2}} dZ^2_s}.
\end{eqnarray*}
Since $Z$ is invariant with respect to random rotations then $Z'$ has the same law of $Z$. Using the previous change of driving semimartingale, we have that $X_t \in \mathbb{R}^2$ solves the following equation
\begin{eqnarray*}
dX^1_{t_-}&=&X^1_{t_-}f(R_{t_-})dt+\frac{X^1_{t_-}}{\sqrt{(X^1_{t_-})^2+(X^2_{t_-})^2}} dZ'^1_t - \frac{X^2_{t_-}}{\sqrt{(X^1_{t_-})^2+(X^2_{t_-})^2}} dZ'^2_t\\
dX^2_{t_-}&=&X^2_{t_-}f(R_{t_-})dt+\frac{X^2_{t_-}}{\sqrt{(X^1_{t_-})^2+(X^2_{t_-})^2}} dZ'^1_t+
\frac{X^1_{t_-}}{\sqrt{(X^1_{t_-})^2+(X^2_{t_-})^2}} dZ'^2_t.
\end{eqnarray*}
If we introduce polar coordinates we obtain
$$dR_t=2R_{t_-}f(R_{t_-})dt+2\sqrt{R_{t_{-}}} dZ'^1_t+ d\mathfrak{Z}'_t$$
where $\mathfrak{Z}'_t = [Z'^1,Z'^1]_t+[Z'^2,Z'^2]_t$. It is important to note that we can perform the previous reduction using polar coordinates only if $Z$ is invariant with respect to random rotations and not just with respect to deterministic rotations. Finally we remark that, if $Z$ is a Brownian motion and $f=0$, the equation for $R_t$ reduces to the usual equation for squared Bessel process.

\section{Gauge symmetries and semimartingales characteristics}\label{section_characteristics}

\subsection{Characteristics of a Lie group valued semimartingale}

In this section we  extend the well known concept of semimartingale characteristics  from the $\mathbb{R}^n$ setting
to the case of a semimartingale defined on a general finite dimensional Lie group $N$. \\
Fixing $n$ generators $Y_1,...,Y_n$ of right-invariant vector fields on $N$, we introduce a set of functions $h^1,...,h^n$ (called \emph{truncated functions related to $Y_1,...,Y_n$}) which are measurable, bounded, smooth in a neighbourhood of
the identity $1_N$, with compact support and such that $h^{\alpha}(1_N)=0$ and $Y_{\alpha}(h^{\beta})(1_N)=\delta^{\beta}_{\alpha}$ (the existence of these functions is proved, for example, in
\cite{Hunt1956} and they can be chosen to be equal to a set of canonical coordinates in a neighbourhood of $1_N$). Generalizing Jacod and Shiryaev \cite[Chapter II Definition 2.6]{Jacod2003} we give the following

\begin{definition}\label{definition_characteristic}
Let $b$ be a predictable semimartingale of bounded variation on $\mathbb{R}^n$ and let $A$ be a predictable continuous semimartingale taking values
in the set of semidefinite positive $n \times n$ matrices. Furthermore, let  $\nu$ be a predictable random measure defined on $\mathbb{R}_+ \times N$.  If
 $Z$ is a semimartingale on a Lie group $N$, we say that $Z$ has characteristics $(b,A,\nu)$ with respect to $Y_1,...,Y_n$ and their truncated functions $h^1,...,h^n$ if,
for any smooth bounded functions $f,g \in \cinf(N)$ and for any smooth and bounded function $p$ which is identically $0$ in a neighbourhood of
$1_N$, we have that
\begin{eqnarray}
&\sum_{0 \leq s \leq t}p(\Delta Z_s)-\int_{0}^t{\int_N{p(z')\nu(ds,dz')}},&\label{equation_characteristic2}\\
&[f(Z),g(Z)]^c_t-g(Z_0)f(Z_0)-\int_0^t{Y_{\alpha}(f)(Z_{s_-})Y_{\beta}(g)(Z_{s_-})dA^{\alpha\beta}_s},&\label{equation_charcteristic3}\\
&\begin{array}{c}
f(Z_t)-f(Z_0)-\int_0^t{Y_{\alpha}(f)(Z_{s_-})db^{\beta}_s}-\frac{1}{2}\int_0^t{Y_{\alpha}(Y_{\beta}(f))(Z_{s_-})dA^{\alpha\beta}_s}+\\
-\sum_{0 \leq s \leq t}(f(Z_s)-f(Z_{s_-})-h^{\alpha}(\Delta
Z_s)Y_{\alpha}(f)(Z_{s_-})),
\end{array}&\label{equation_characteristic4}
\end{eqnarray}
 where $\Delta Z_t=Z_t \cdot Z_{t_-}^{-1}$, are local martingales.
\end{definition}

\begin{remark}
We note that condition
\refeqn{equation_charcteristic3} is redundant, because it can be
deduced from  \refeqn{equation_characteristic2} and
\refeqn{equation_characteristic4}.
\end{remark}

\noindent The following theorem states that any semimartingale $Z$ defined on a Lie group $N$ admits  essentially a unique  characteristic triplet $(b,A,\nu)$.

\begin{theorem}\label{theorem_characteristic1}
If $Z$ is a semimartingale on a Lie group $N$, then $Z$ admits a
characteristic triplet $(b,A,\nu)$ with respect to $Y_1,...,Y_n$ and
$h^1,...h^n$, which is  unique up to
$\mathbb{P}$ null sets.
\end{theorem}
\begin{proof}
We first prove the existence.
Given a  semimartingale $Z$ on $N$, we can associate with $Z$ a unique random
measure on $N$ given by
$$\mu^Z(\omega,dt,dz)=\sum_{s \geq 0}I_{\Delta Z_s \not =1_N}\delta_{(s,\Delta Z_s(\omega))}(dt,dz),$$
where $\delta_a$ is the Dirac delta with mass in $a \in \mathbb{R}_+ \times N$. The random measure
$\mu^Z$ is an integer-valued random measure (see, e.g., \cite[Chapter II, Proposition 1.16]{Jacod2003}),
hence there exists a unique non-negative predictable random measure $\mu^{Z,p}$, which is the
compensator of $\mu^Z$ i.e. such that for any predictable random function $h$ vanishing in a neighbourhood of $1_N$, $\int_0^t{\int_N{h(z')\mu^Z(ds,dz')}}-\int_0^t{\int_N{h(z')\mu^{Z,p}(ds,dz')}}$ is a martingale, see \cite[Chapter II Theorem 1.8]{Jacod2003}, and so we can take $\nu=\mu^{Z,p}$.\\
In order to prove the existence of processes $b^{\alpha},A^{\alpha\beta}$ we  introduce a Riemannian embedding $S:N \rightarrow
\mathbb{R}^{k_N}$ with respect to a left invariant metric on $N$ (the existence of such an embedding for general finite dimensional Riemannian manifolds is provided by the Nash theorem, see for example \cite{Han2006}). Put $Z^{i}=S^i(Z)$, where $S=(S^1,...S^{k_N})$, and write
$$\hat{Z}^i_t=Z^i_t-\sum_{0 \leq s \leq t}\left(\Delta Z^i_s- h^{\alpha}(\Delta Z_s)
Y_{\alpha}(S^i)(Z_{s_-})\right).$$
Since $S$ is Riemannian, the norms of  $S_*(Y_{\alpha})(x)$
are constant and so $Y_{\alpha}(S^i)$ are bounded. Because of
$$\Delta \hat{Z}^i_t= h^{\alpha}(\Delta Z_t)Y_{\alpha}(S^i)(Z_{t_-}),$$
and $h^{\alpha}$  being  bounded, $\hat{Z}^i$ have bounded jumps and
so they are special semimartingales (see \cite[Chapter III Theorem 31]{Protter1990}). By definition of special semimartingale (see \cite[Chapter III Section 8]{Protter1990}), the processes
$\hat{Z}^i$ can be decomposed in a unique way as
$$\hat{Z}^i=B^i+M^{i,c}+M^{i,d},$$
where $B^i$ is a predictable process of bounded variation, $M^{i,c}$ is a continuous local martingale and $M^{i,d}$ is a purely
discontinuous local martingale. \\

\noindent Therefore, we can define
\begin{eqnarray*}
b^{\alpha}_t&=&\int_0^t{\tilde{P}^{\alpha}_i(Z_{s_-})dB^i_s+Y_{\beta}(\tilde{P}^{\alpha}_i)(Z_{s_-})\tilde{P}^{\beta}_j(Z_{s_-})d[M^{i,c},M^{j,c}]_s}\\
A^{\alpha\beta}_t&=&\int_0^t{\tilde{P}^{\alpha}_i(Z_{s_-})\tilde{P}^{\beta}_j(Z_{s_-})d[M^{i,c},M^{j,c}]_s},
\end{eqnarray*}
where $\tilde{P}$ is the pseudo inverse of the matrix $P=(Y_{\alpha}(S^i))$.
Given $f,g \in \cinf(N)$ let us consider two extensions $\tilde{f},\tilde{g}$ on $\mathbb{R}^{k_N}$ of $f,g \in \cinf(N)$ as in Lemma \ref{remark_extension}:
by It\^o formula we have
\begin{equation}\label{equation_characteristic1}
\begin{array}{rcr}
f(Z_t)-f(Z_0)&=&\int_0^t{\partial_{s^i}(\tilde{f})(Z_{s_-})dZ^i_s}+\frac{1}{2} \int_0^t{\partial_{s^is^j}(\tilde{f})(Z_{s_-})d[Z^i,Z^j]^c_s}+\\
&&+\sum_{0 \leq s \leq t}(f(Z_s)-f(Z_{s_-})-\Delta {Z}^i_s
\partial_{s^i}(\tilde{f})(Z_{s_-}))
\end{array}
\end{equation}
and the corresponding formula holds for $g$. Recalling that
$[\tilde{Z}^i,\tilde{Z}^j]^c=[Z^i,Z^j]^c=[M^{i,c},M^{j,c}]$ and using equation \refeqn{equation_extension1}, we have

$$[f(Z),g(Z)]^c_t=\int_0^t{Y_{\alpha}(f)(Z_{s_-})Y_{\beta}(g)(Z_{s_-})dA^{\alpha\beta}_s}.$$
Finally, recalling that
\begin{eqnarray*}
Z^i_t=B^i_t+M^{i,c}_t+M^{i,d}_t+\sum_{0 \leq s \leq t}(\Delta Z_s^i-h^{\alpha}(\Delta Z_s)Y_{\alpha}(S^i)(Z_{s_-}))
\end{eqnarray*}
and using equation \refeqn{equation_extension2}, equation \refeqn{equation_characteristic1} and Lemma
\ref{lemma_geometrical1} we obtain that
\begin{eqnarray*}
&f(Z_t)-f(Z_0)-\int_0^t{Y_{\alpha}(f)(Z_{s_-})db^{\alpha}_s}-\frac{1}{2}\int_0^t{Y_{\alpha}(Y_{\beta}(f))(Z_{s_-})dA^{\alpha\beta}_s}+&\\
&-\sum_{0 \leq s \leq t}(f(Z_s)-f(Z_{s_-})-h^{\alpha}(\Delta
Z_s)Y_{\alpha}(f)(Z_{s_-}))&
\end{eqnarray*}
is a local martingale.\\
The uniqueness of $\nu$ has already been proved using the uniqueness of the compensator of the random measure $\mu^Z$  (see \cite[Chapter II,
Theorem 1.8]{Jacod2003}). The uniqueness of 
$b^{\alpha},A^{\alpha\beta}$ follows in a standard way (see \cite[Chapter 2 Remark 2.8]{Jacod2003}) exploiting the fact that 
martingales of bounded variation are constant (see, e.g., \cite[Chapter III, Theorem 12]{Protter1990}).
\hfill\end{proof}

\subsection{Gauge symmetries and semimartingales characteristics}\label{subsection_gauge_main}

In this section, after introducing some useful geometric and probabilistic tools, we look for suitable conditions on the characteristics of a
semimartingale in order to ensure that it admits a gauge symmetry group.\\

\noindent We start by discussing the role of the filtration $\mathcal{F}_t$ in Definition \ref{definition_gauge}: in fact, although
the definition of gauge symmetry group seems to concern only the law of $Z$ and not the  filtration, a
semimartingale $Z$ may admit a gauge symmetry group $\mathcal{G}$ with respect to a filtration $\mathcal{F}_t$  such that $\mathcal{G}$ is
no longer a gauge symmetry group for $Z$ if a different filtration $\mathcal{H}_t$ is considered. For example, let $W$ be a standard $n$
dimensional Brownian motion, let $\mathcal{F}_t$ be its natural filtration and  put $\mathcal{H}_t=\mathcal{F}_t \vee \sigma(W_T)$.
Although $W$ is a semimartingale with respect to both $\mathcal{F}_t$ and $\mathcal{H}_t$ (see, e.g. \cite[Chapter VI Theorem 3]{Protter1990}), the rotations are a gauge symmetry group for $W$ only with
respect to the filtration $\mathcal{F}_t$ and not with respect to $\mathcal{H}_t$. In fact if $B:\mathbb{R}^n \rightarrow O(n)$ is a measurable
map such that $B(x) \cdot x=(|x|,0,...,0)$, the constant process $B(W_T)$ is predictable with respect to the filtration $\mathcal{H}_t$ and it
is not adapted with respect to $\mathcal{F}_t$. On the other hand the semimartingale
$$\tilde{W}^{\alpha}_t=\int_0^t{B^{\alpha}_{\beta}(W_T)dW^{\beta}_s}=B^{\alpha}_{\beta}(W_T) W^{\beta}_t,$$
is not a Brownian motion since, for example, $\tilde{W}_T=(|W_T|,0,...,0)$ is not a Gaussian random variable. This is due to the fact
that the family of the $\mathcal{H}_t$-predictable processes is too large for preserving the invariance property of Brownian motion. In order to
avoid this kind of phenomena, and ensuring that owning a gauge symmetry is a property of the law of the process $Z$ and not of
its filtration, we introduce the following definition.

\begin{definition}
Let $Z$ be a semimartingale with respect to the filtration $\mathcal{F}_t$. We say that the filtration $\mathcal{F}_t$ is a generalized natural
filtration if there exists a version of the characteristic triplet $(b,A,\nu)$ of $Z$ (with respect to the filtration $\mathcal{F}_t$), which is
predictable with respect to the natural filtration $\mathcal{F}^Z_t \subset \mathcal{F}_t$ of the semimartingale $Z$.
\end{definition}

It is important to note that if $(b,A,\nu)$ are the characteristics of a semimartingale $Z$ with respect to its natural filtration, then they
are also the characteristics of $Z$ with respect to any generalized natural filtration for $Z$. For this reason, hereafter, whenever we consider
a generalized natural filtration $\mathcal{F}_t$ for $Z$ we can use the characteristics $(b,A,\nu)$ with respect to the natural filtration of
$Z$ as the characteristics of $Z$ with respect to $\mathcal{F}_t$.\\

Let us  consider the probability space
$$\Omega^c=\Omega_A \times \Omega_B,$$
where $\Omega_A=\mathcal{D}_{1_N}([0,+\infty),N)$ is the space of c\acc{a}dl\acc{a}g functions $\omega_A(t)$ taking values on $N$ and such that
$\omega_A(0)=1_N$, and $\Omega_B=L_{loc}^{\infty}([0,+\infty),\mathcal{G})$ is
the set of locally bounded and measurable functions taking values in $\mathcal{G}$.\\
On the set $\Omega_A$ we consider the standard filtration $\mathcal{F}^A_{t}$ of $\mathcal{D}_{1_N}([0,+\infty),N)$ and on $\Omega_B$  the
filtration $\mathcal{F}^B_t$ generated by the standard filtration of $C^0([0,+\infty),\mathcal{G}) \subset \Omega_B$ (usually called the
predictable filtration). We denote by $\pi_A, \pi_B$ the projections of $\Omega$ onto $\Omega_A$ and $\Omega_B$ respectively, and so we define
$\mathcal{F}^c_t=\sigma(\pi_A^{-1}(\mathcal{F}^A_t),\pi^{-1}_B(\mathcal{F}^B_t))$. We call $\Omega^c$ the canonical probability space and
$\mathcal{F}^c_t$ the natural filtration on $\Omega^c$. \\
We need the space $\Omega_A$ in order  to define a semimartingale $Z$ on $N$, and the space $\Omega_B$ in order  to define a locally bounded predictable process
taking values on $\mathcal{G}$. Choosing a particular semimartingale $Z$ on $N$ and a predictable process $G_t$ on $\mathcal{G}$ is equivalent to
fixing a probability measure $\mathbb{P}$ on $\Omega$ such that $Z_t(\omega)=\pi_A(\omega)(t)$ is a semimartingale on $N$ (the fact that the
process $G_t(\omega)=\pi_B(\omega)(t)$ is a locally bounded predictable process
is automatically guaranteed by the choices of the space $\Omega_B$ and the filtration $\mathcal{F}^B_t$).\\
Given  an $N$ valued semimartingale $Z$ and  a generic predictable process $G_t$  taking values in $\mathcal{G}$, both defined on a
probability space $(\Omega,\mathcal{F},\mathcal{F}_t,\mathbb{P})$, there exists a natural
probability measure $\mathbb{P}^c$ on the canonical probability space $\Omega^c$ induced by the probability measure $\mathbb{P}$ such that the canonical processes $(\omega_A(t),\omega_B(t))$ have the same law as $(Z_t(\omega),G_t(\omega))$. 
Thus, fixing the process
$G_t$ and the law $\mathbb{P}^Z$ of the semimartingale $Z_t$ is equivalent to fixing the probability law $\mathbb{P}^c$ on $\Omega^c$ so that
the restriction of $\mathbb{P}^c$ to the $\Omega_A$ measurable subsets,  $\mathbb{P}=\mathbb{P}^c|_{\mathcal{F}^A}$, is exactly $\mathbb{P}^Z$.
As a consequence, proving a statement  involving only the measurable objects $Z_t,G_t$ which turns out to be independent from the choice of a specific
predictable process $G_t$, is equivalent to proving the same statement on the probability space $\Omega^c$ with respect to the canonical
processes $\omega_A(t),\omega_B(t)$ and for a suitable subset of probability laws $\mathbb{P}^c$ on $\Omega^c$ such that
$\mathbb{P}|_{\mathcal{F}^A}=\mathbb{P}^Z$. This subset depends on the filtration  $\mathcal{F}_t$ of the probability space chosen. In
particular if $\mathcal{F}_t$ is a generalized natural filtration for $Z$, then $\tilde{\mathcal{F}}^c_t$ is a generalized natural filtration for
$\omega_A(t)$ (where $\tilde{\mathcal{F}}^c_t$ is the completion of $\mathcal{F}^c_t$ with respect to $\mathbb{P}^c$). Since we consider only
generalized natural filtrations for the semimartingale $Z$, we suppose that $\mathbb{P}^c$ is such that $\tilde{\mathcal{F}}^c_t$ is a generalized natural
filtration.\\
For this reason, in the following  we only consider  the canonical probability space $\Omega^c$ with law $\mathbb{P}=\mathbb{P}^c$ and denote by $Z_t$
the canonical semimartingale $\omega_A(t)$ and by $G_t$ the canonical predictable process $\omega_B(t)$.\\
In the same way, we identify  the solution $\tilde{Z}$ to the SDE $d\tilde{Z}_t=\Xi_{G_t}(dZ_t)$ with the measurable map
$\Lambda_A:\Omega \rightarrow \Omega_A$ such that $\tilde{Z}_t(\omega)=\Lambda_{A}(\omega)(t)$. We can extend the map $\Lambda_A$ to a map
$\Lambda:\Omega \rightarrow \Omega$ given by
$$\Lambda(\omega)=(\Lambda_A(\omega),\pi_2(\omega)),$$
defining a new probability measure
$\mathbb{P}'=\Lambda_*(\mathbb{P})$. The map $\Lambda$  is $\mathbb{P}$ invertible, i.e. there exists a map $\Lambda'$ such that
$\Lambda \circ \Lambda'$ is equal to the identity map up to $\mathbb{P}'$ null sets and the map $\Lambda' \circ \Lambda$ is equal to the
identity map up to $\mathbb{P}$ null sets. The construction of the  map $\Lambda'$ is similar to the construction of $\Lambda$ starting from  the stochastic
differential equation $dZ_t=\Xi_{(G_t)^{-1}}(d\tilde{Z}_t)$ and the measure $\mathbb{P}'$. The proof of the fact that $\Lambda'$ is the
$\mathbb{P}'$ inverse of $\Lambda$ and hence $\Lambda$ is the $\mathbb{P}$ inverse
of $\Lambda'$, is based on Theorem \ref{theorem_gauge1}. It is important to note that $\hat{\mathcal{F}}^c_t$, i.e.
the completion of $\mathcal{F}^c_t$ with respect to the probability $\mathbb{P}'$, may not be a generalized natural
filtration for $\omega_A(t)$ even if $\tilde{\mathcal{F}}^c_t$ is a generalized natural filtration for $Z_t$ under $\mathbb{P}$. \\
Given the probability law $\mathbb{P}^Z$ on $\Omega_A$, by Theorem \ref{theorem_characteristic1} there exist some measurable and predictable
functions $b^{\alpha},A^{\alpha\beta}:\Omega_A \times \mathbb{R}_+ \rightarrow \mathbb{R}$ and a random predictable measure $\nu:\Omega_A
\rightarrow \mathcal{M}(\mathbb{R}_+ \times N)$ which are the characteristics of the canonical process $Z_t(\omega)=\pi_A(\omega(t))$ and are
uniquely defined up to $\mathbb{P}^Z$ null-sets. The characteristic triplets $(b,A,\nu)$, seen as $\mathcal{F}^A$ measurable objects, are
uniquely determined by the probability measure $\mathbb{P}^Z$. The converse, namely the fact  that the $\mathcal{F}^A$ measurable objects
$(b,A,\nu)$ uniquely identify the probability law $\mathbb{P}^Z$, is in general not true (the reader can think, for example, to  diffusion
processes whose martingale problem admits multiple solutions). When the triplet $(b,A,\nu)$ uniquely determines the probability
law $\mathbb{P}^Z$ on $\Omega_A$ we say that the triplet $(b,A,\nu)$ \emph{uniquely identifies the law of $Z$} (the reader interested in this problem and in the more general martingale problem for semimartingales is invited to consult \cite[Chapter III]{Jacod2003}). Examples of this situation  are,
e.g., the $\mathbb{R}^n$ Brownian motion, $\mathbb{R}^n$ L\'evy processes, diffusion processes with a unique solution to the associated
martingale problem, and point processes. If the law $\mathbb{P}$ on $\Omega^c$ is such that $\mathbb{P}|_{\mathcal{F}^A}=\mathbb{P}^Z$ and the
filtration $\tilde{\mathcal{F}}^c_t$ is a generalized natural filtration for $\omega_A(t)$, then the same $\mathcal{F}^A$ measurable
characteristics $(b,A,\nu)$, viewed as $\Omega^c$ semimartingales,  are  characteristics of $Z$ with respect to $\tilde{\mathcal{F}}^c_t$ as
well. Obviously it is possible to define other characteristic triplets $(\bar{b},\bar{A},\bar{\nu})$ of $Z$ on $\Omega^c$ which are only
$\tilde{\mathcal{F}}^c$ adapted and not $\mathcal{F}^A_t$ adapted. The characteristics
$(\bar{b},\bar{A},\bar{\nu})$ are equal to $(b,A,\nu)$ up to $\mathbb{P}$ null sets (and not only up to $\mathbb{P}^Z$ null sets).\\

Since the map $\Xi_g$ is such that $\Xi_g(1_N)=1_N$, using the identification of the Lie algebra $Y_1,...,Y_n$ generated by the right-invariant vector fields with $\mathfrak{n}=T_{1_N}N$ we can define two linear maps $\Gamma_g:\mathfrak{n} \rightarrow \mathfrak{n}$ and $O_g:\mathfrak{n} \otimes \mathfrak{n} \rightarrow \mathfrak{n}$ in the following way: if $Y,Y'$ are two invariant vector fields with associated flows $\Phi_a,\Phi'_b:N \rightarrow N$ we define
\begin{eqnarray*}
\Gamma_g(Y)&=&\partial_a(\Xi_g \circ \Phi_a)(1_N)|_{a=0}\\
O_g(Y',Y)&=&\partial_b(\partial_a(\Xi_g \circ \Phi_a \circ \Phi'_b))(1_N)|_{a=0,b=0}.
\end{eqnarray*}
Since $Y(f)(z)=\partial_a(f\circ \Phi_a)(z)|_{a=0}$ and the flow of right-invariant vector fields commutes with the right multiplication, we can exploit the definition of $\Psi^{\Xi}_g(\tilde{z},z)=\Xi_g(z) \cdot \tilde{z}$ and $\overline{\Psi}^{\Xi_g}(\tilde{z},z',z)=\Xi_g(z' \cdot z^{-1}) \cdot \tilde{z}$ to prove that, for any $Y,Y'$ right-invariant vector fields,
\begin{eqnarray}
&\begin{array}{c}
Y^z(f \circ \Psi^{\Xi}_g)(\tilde{z},1_N)=Y^{z'}(f \circ \overline{\Psi}^{\Xi}_g)(\tilde{z},z,z)=\\
=(\Gamma_g(Y))(f)(\tilde{z})
\end{array}&\label{equation_upsilon}\\
&\begin{array}{c}
Y'^z(Y^z(f \circ \Psi^{\Xi}_g))(\tilde{z},1_N)=Y'^{z'}(Y^{z'}(f \circ \overline{\Psi}^{\Xi}_g))(\tilde{z},z,z)=\\
=(\Gamma_g(Y'))[(\Gamma_g(Y))(f)](\tilde{z})+(O_g(Y',Y))(f)(\tilde{z}),
\end{array}&\label{equation_O}
\end{eqnarray}
where the superscript $\cdot^z,\cdot^{z'}$ means that the vector fields $Y,Y'$ apply to the $z,z'$ variables respectively.

\begin{theorem}\label{theorem_characteristic2}
Let $Z$ be a semimartingale on a Lie group $N$ with characteristic triplet $(b(\omega_A),A(\omega_A),\nu(\omega_A))$. Suppose that $Z$ admits
$\mathcal{G}$ with action $\Xi_g$ as a gauge symmetry group with respect to any generalized natural filtration. Then if $\mathbb{P}$ is a measure
on $\Omega^c$ such that $\tilde{\mathcal{F}}_t$ is a generalized natural filtration with respect to both $Z_t$ and $d\tilde{Z}_t=\Xi_{G_t}(dZ_t)$, we have
\begin{equation}\label{equation_characteristic8}
\begin{array}{rcl}
db^{\alpha}_t(\omega)&=&\Gamma^{\alpha}_{g(\omega_B),\beta}db_t^{\beta}(\pi_A(\Lambda'(\omega)))+\frac{1}{2}
O^{\alpha}_{g(\omega_B),\beta\gamma}dA^{\beta\gamma}_t(\pi_A(\Lambda'(\omega)))+\\
&&+\int_N{(h^{\alpha}(z')-h^{\beta}(\Xi_{g^{-1}(\omega_B)}(z'))\Gamma^{\alpha}_{g(\omega_B),\beta})\nu(\pi_A(\Lambda'(\omega)),dt,dz')}
\end{array}
\end{equation}
\begin{eqnarray}
dA_t^{\alpha\beta}(\omega)&=&\Gamma^{\alpha}_{g(\omega_B),\gamma}\Gamma^{\beta}_{g(\omega_B),\delta}dA_t^{\gamma\delta}
(\pi_A(\Lambda'(\omega)))\label{equation_characteristic9}\\
\nu(\omega,dt,dz)&=&\Xi_{g(\omega_B)*}(\nu(\pi_A(\Lambda'(\omega)),dt,dz)),\label{equation_characteristic10}
\end{eqnarray}
up to a $\mathbb{P}'=\Lambda_*(\mathbb{P})$ null set. Furthermore, if $\tilde{b},\tilde{A},\tilde{\nu}$ are $\pi^{-1}_A(\mathcal{F}^A)$
measurable, the previous equalities hold with respect to $\mathbb{P}^Z$ null sets. Finally, if $(b,A,\nu)$ uniquely determines the law of
$Z$, the previous conditions are also sufficient for the existence of a gauge symmetry group.
\end{theorem}

Before proving the theorem we study the transformations of the characteristics under (random) semimartingale changes.

\begin{lemma}\label{lemma_characteristic1}
If $Z$ is a semimartingale with characteristics $(b,A,\nu)$, then $d\tilde{Z}=\Xi_{G_t}(dZ)$ is a semimartingale with characteristics

\begin{eqnarray*}
&d\tilde{b}^{\alpha}_t=\Gamma^{\alpha}_{G_t,\beta}db^{\beta}_t+\frac{1}{2}O^{\alpha}_{G_t,\beta\gamma}dA^{\beta\gamma}_t
+\int_N{(h^{\alpha}(z')-h^{\beta}(\Xi_{G_t^{-1}}(z'))\Gamma^{\alpha}_{G_t,\beta})\nu(dt,dz')}&\\
&d\tilde{A}^{\alpha\beta}_t=\Gamma^{\alpha}_{G_t,\gamma}\Gamma^{\beta}_{G_t,\delta}dA^{\gamma\delta}_t&\\
&\tilde{\nu}=\Xi^*_{G_t}(\nu).&
\end{eqnarray*}
\normalsize
\end{lemma}
\begin{proof}
We denote by $(\tilde{b},\tilde{A},\tilde{\nu})$ the characteristic triplet of $\tilde{Z}$. Since the jumps of $\tilde{Z}$ are $\Delta
\tilde{Z}_t=\Xi_{G_t}(\Delta Z_t)$ and the jump times of $\tilde{Z}$ are the same of $Z$  we have, using the notation of  Theorem
\ref{theorem_characteristic1},
$$\mu^{\tilde{Z}}(dt,dz)=\sum_{s \geq 0}I_{\Delta \tilde{Z}_s \not =0}(s)\delta_{(s,\Delta \tilde{Z}_s)}(dt,dz)=\sum_{s \geq 0}
I_{\Delta Z_s \not =0}(s)\delta_{(s,\Xi_{G_s}(\Delta Z_s))}(dt,dz).$$
If we identify, with a slight abuse of notation, the push-forward of the
map $(s,z) \rightarrow (s,\Xi_{G_s}(z))$ with the push-forward of the map $(s,z) \rightarrow \Xi_{G_s}(z)$, we have
$$\delta_{(s,\Xi_{G_s}(\Delta Z_s))}(du,dz)=\Xi_{G_s,*}(\delta_{(s,\Delta Z_s)})(du,dz),$$
and so $\mu^{\tilde{Z}}=\Xi_{G_t,*}(\mu^Z)$.
If we consider a function $h:N \rightarrow \mathbb{R}$ which is identically zero in a neighbourhood of $1_N$, by definition of push-forward of a measure we have
$$\int_0^t{\int_N{h(z)\Xi_{G_s,*}(\mu^Z-\nu)(ds,dz)}}=\int_0^t{\int_N{h(\Xi_{G_s}(z))(\mu^Z-\nu)(ds,dz)}}.$$
Furthermore $\int_0^t{\int_N{h(\Xi_{G_s}(z))(\mu^Z-\nu)(ds,dz)}}$ is a martingale, since $h(\Xi_{G_s}(z))$ is a predictable function and $\nu$
is the predictable projection of the random measure $\mu^Z$. Since $\mu^{\tilde{Z}}=\Xi_{G_t,*}(\mu^Z)$ we have that $\Xi_{G_t,*}(\nu)$ is the
predictable projection of the measure $\mu^{\tilde{Z}}$ and $\tilde{\nu}=\Xi_{G_t,*}(\nu)$.\\
For the formulas of $\tilde{A}$ and $\tilde{b}$ we use the definition of solution to a geometrical SDE, Lemma \ref{lemma_geometrical1} and equation \refeqn{equation_upsilon} and \refeqn{equation_O} for $\Gamma_g$ and $O_g$. We make the proof only for $\tilde{A}$, since the proof for $\tilde{b}$ follows the same lines. \\
Fixing an embedding $S:N \rightarrow \mathbb{R}^{k_N}$, by Lemma \ref{lemma_geometrical1}, for any functions
$f,g \in \cinf(N)$, equation \refeqn{equation_upsilon} for $\Gamma_g$ ensures that
\normalsize
$$
\begin{array}{c}
d[f(\tilde{Z}),g(\tilde{Z})]_t^c=
\partial_{s'^i}(\tilde{f} \circ \overline{\Psi}^{\Xi}_g)(\tilde{Z}_{t_-},Z_{t_-},Z_{t_-})
\partial_{s'^j}(\tilde{g}\circ \overline{\Psi}^{\Xi}_g)(\tilde{Z}_{t_-},Z_{t_-},Z_{t_-})\cdot\\
\cdot d[S^i(Z),S^j(Z)]_t^c=Y^{z'}_{\alpha}(\tilde{f} \circ \overline{\Psi}^{\Xi}_g)(\tilde{Z}_{t_-},Z_{t_-},Z_{t_-})\cdot\\
\cdot Y^{z'}_{\beta}(\tilde{g} \circ \overline{\Psi}^{\Xi}_g)(\tilde{Z}_{t_-},Z_{t_-},Z_{t_-})\tilde{P}^{\alpha}_i(Z_{t_-})
\tilde{P}^{\beta}_j(Z_{t_-})d[S^i(Z),S^j(Z)]^c_t\\
=Y_{\gamma}(f)(\tilde{Z}_{t_-})Y_{\delta}(g)(\tilde{Z})\Gamma^{\gamma}_{G_t,\alpha}\Gamma^{\delta}_{G_t,\beta}\tilde{P}^{\alpha}_i(Z_{t_-})\tilde{P}^{\beta}_j(Z_{t_-})
d[S^i(Z),S^j(Z)]^c_t,
\end{array}
$$
\normalsize
where $\tilde{g},\tilde{f}$ are the usual extensions of $f,g$ on
$\mathbb{R}^{k_N}$, and $\tilde{P}$ is a pseudoinverse matrix of
$P=(Y_{\alpha}(S^i))$ (see Lemma \ref{remark_extension}).
By definition of characteristics we have that
\begin{eqnarray*}
&[S^i(Z),S^j(Z)]_t^c-\int_0^t{Y_{\alpha}(S^i)(Z_{s_-})Y_{\beta}(S^j)(Z_{s_-})dA^{\alpha\beta}_s}=&\\
&=[S^i(Z),S^j(Z)]_t^c-\int_0^t{P^i_{\alpha}(Z_{s_-})P^j_{\beta}(Z_{s_-})dA^{\alpha\beta}_s}&
\end{eqnarray*}
is a local martingale. This means that
$$[f(\tilde{Z}),g(\tilde{Z})]^c_t-\int_0^t{Y_{\gamma}(f)(\tilde{Z}_{s_-})Y_{\delta}(g)(\tilde{Z}_{s_-})\Gamma^{\gamma}_{G_s,\alpha}\Gamma^{\delta}_{G_s,\beta}dA^{\alpha\beta}_s}$$
is a local martingale. ${}\hfill$
\hfill\end{proof}

\noindent\begin{proof}[Proof of Theorem \ref{theorem_characteristic2}]
We cannot directly use  Lemma \ref{lemma_characteristic1} to compare $(b,A,\nu)$ with $(\tilde{b},\tilde{A},\tilde{\nu})$, since
 $Z$ and $\tilde{Z}$, where $d\tilde{Z}_t=\Xi_{G_t}(dZ_t)$, are two different processes being two different
 functions  from $\Omega^c \times \mathbb{R}_+$ into $N$. Indeed $Z_t(\omega)=\pi_A(\omega)(t)$,
while $\tilde{Z}_t(\omega)=\pi_A(\Lambda(\omega))(t)$.\\
Since $\Lambda'$ is the $\mathbb{P}'$ inverse of $\Lambda$, $\tilde{Z}(\Lambda'(\omega))$ is exactly the same process as $Z$ (as functions
defined on $\Omega^c$). If $\tilde{Z}(\Lambda')$ and $Z$ have the same law, and since both the filtrations $\hat{\mathcal{F}}^c_t$ and $\tilde{\mathcal{F}}^c_t$ are canonical, they necessarily have the same characteristics up to a $\mathbb{P}'$
null set and therefore $b(\omega)=\tilde{b}(\Lambda'(\omega))$, $A(\omega)=\tilde{A}(\Lambda'(\omega))$ and
$\nu(\omega)=\tilde{\nu}(\Lambda'(\omega))$. If $\tilde{b}(\Lambda')$, $\tilde{A}(\Lambda'(\omega))$ and $\tilde{\nu}(\Lambda'(\omega))$ are
$\pi_A^{-1}(\mathcal{F}^A_t)$ measurable
(usually they are only $\hat{\mathcal{F}}^c_t$ measurable) they are then equal to $b,a$ and $\nu$ up to a null set with respect to
 $\pi_{A*}(\mathbb{P})=\pi_{A*}(\mathbb{P}')$.\\
Obviously if $(b,A,\nu)$ uniquely  identifies the law of $Z$ in $\Omega_A$, the condition stated in the theorem is also sufficient.${}\hfill$
\hfill\end{proof}

\subsection{An example of gauge symmetries for discrete semimartingales}

Although Theorem \ref{theorem_characteristic2} is very useful for proving that a group \emph{is not} a gauge symmetry group for a given semimartingale, in general it does not suffice by itself to delineate a method for determining the gauge symmetry group of a given semimartingale. However, in the following we attract the reader's attention to some still general example, although in discrete time, where this determination is indeed possible.\\
We consider the case of a stochastic process with discrete time on Lie groups $N$. In this case it is possible to use Theorem \ref{theorem_characteristic2} for obtaining a complete characterization of gauge invariant semimartingales. In order to illustrate the idea we limit ourselves to a detailed treatment only for the case where $N=\mathbb{R}^n$ (with its natural Lie group multiplication given by the sum of vectors) and the gauge group is $O(n)$ with its natural action $\Xi_B=B \cdot z$ on $\mathbb{R}^n$. The case of a general group $N$ can be treated in a similar way, but we prefer to skip the details in this paper. \\ 

\noindent Let $(Z_0=0,Z_1,...,Z_n,...)$ be a discrete time stochastic process taking values on $\mathbb{R}^n$. We can identify this process with a continuous time semimartingale $Z_t$, $t\in[0,+\infty)$, in the following way 
$$Z_t=Z_n \text{ if }n-1\leq t<n, \ n\in\mathbb{N} .$$
The stochastic characteristics of the semimartingale $Z$ are thus of the form
$$b_t=0, \ \ \ \ A_t=0, \ \ \ \ \nu(dz,dt)=\sum_{n \in \mathbb{N}} \delta_n(dt) \mu_n(Z_0,Z_1,...,Z_{n-1},dz),$$
where $z\in\mathbb{R}^n$ and $\mu_n(Z_0,Z_1,...,Z_{n-1},dz)$ is the law of the jump $\Delta Z_n=Z_n -Z_{n-1}$ conditioned with respect to the random variables $(Z_0,Z_1,...,Z_{n-1})$.\\ 
In order to simplify the treatment of the problem we suppose that the random measures $\mu_n$ are absolutely continuous with respect to the Lebesgue measure $dz$, this means that there exists a sequence of functions $F_n\in L^1(\mathbb{R}^n,dz)$ such that 
\begin{equation}\label{equation_F}
\mu_n(Z_0,Z_1,...,Z_{n-1},dz)=F_n(\Delta Z_1,...,\Delta Z_{n-1},z) dz.
\end{equation}
Let $B_t \in O(n)$ be a predictable process with respect to the $\sigma$-algebra generated by $Z$. This means that $B$ can be identified by with a discrete time process of the form $B_1:=B_1(Z_0)$, $B_2:=B_2(Z_0,Z_1)$, ...,$B_n:=B_n(Z_0,...,Z_{n-1})$.\\

\noindent The transformed semimartingale $dZ'_t=\Xi_{B_t}(dZ_t)$ and the transformation map $\Lambda$ in this case is given by
$$Z'_n(\omega)=\Lambda(\omega)_n=\sum_{i \leq n}B_i(Z_0(\omega),...,Z_{i-1}(\omega))\cdot (Z_i(\omega)-Z_{i-1}(\omega))$$
and $Z'_0=0$. It is important to note that the expression $\Lambda_n$ depends only on $Z_0,...,Z_n$ and is linear with respect to $Z_n$. This means that, in this case, the map $\Lambda$ is invertible and we can compute explicitly its inverse that is given by 
$$Z_n(\omega)=\Lambda'_n(\omega)=\sum_{i \leq n}\tilde{B}_i(Z'_0(\omega),...,Z'_{i-1}(\omega))\cdot (Z'_i(\omega)-Z'_{i-1}(\omega)),$$
where 
$$\tilde{B}_i(Z'_0(\omega),...,Z'_{i-1}(\omega)):=B_i^{-1}(\Lambda'_0(\omega),...,\Lambda'_{i-1}(\omega)).$$ 
This is obtained replacing recursively the random variable $Z_0,...,Z_{i-1}$ by their expression in terms of the $\Lambda'_0,...,\Lambda'_{i-1}$. It is important to note that in this case the particular structure of the discrete time process of $Z$ and of its predictable $\sigma$-algebra come to help.\\

\noindent The previous analysis implies that $\Delta Z_n=\tilde{B}_n(Z'_0,...,Z'_{n-1}) \cdot \Delta Z'_n$. So applying Theorem \ref{theorem_characteristic2} to this case we get that $Z$ is invariant with respect to random rotations if and only if 
\begin{eqnarray*}
&\nu(dz,dt)=\Xi_{B_n}(\nu(dz,dt)) \circ \Lambda' = &\\
&=\sum_n \delta_n(dt) F_n(\tilde{B}^{-1}_1(Z'_0) \cdot \Delta Z'_1,...,\tilde{B}^{-1}(Z_0,Z_1,...,Z_{n-2})\cdot \Delta Z'_{n-1}, \tilde{B}_n(Z'_0,...,Z'_{n-1}) \cdot z) dz,&
\end{eqnarray*}
where the $F_n$ are given by equation \refeqn{equation_F}. From the previous equation it follows that 
$$F_n(\tilde{B}^{-1}_1(Z'_0) \cdot \Delta Z'_1,...,\tilde{B}^{-1}(Z'_0,Z'_1,...,Z'_{n-2}) \cdot \Delta Z'_{n-1}, \tilde{B}_n(Z'_0,...,Z'_{n-1}) \cdot z)=F_n(\Delta Z'_1,...,\Delta Z'_{n-1},z),$$
for any measurable function $\tilde{B}_i:\mathbb{R}^{n\times i}\rightarrow O(n)$. Since the functions $\tilde{B}_i$ are arbitrary it follows that $Z$ is invariant with respect to random rotations if and only if $F_n$ depends only on the moduli $|\Delta Z_i|$, for $i=1,...,n-1$, and $|z|$, i.e. there exists a positive measurable function $G_n:\mathbb{R}^n_+ \rightarrow \mathbb{R}_+$ such that  
$$F_n(\Delta Z_1,...,\Delta Z_{n-1},z)=G_n(|\Delta Z_1|,...,|\Delta Z_{n-1}|,|z|),$$
where $|\cdot|$ is the usual Euclidean norm in $\mathbb{R}^n$. In other words the stochastic characteristics of $Z$ is of the form
$$\nu(dz,dt)=\sum_{n \in \mathbb{N}} \delta_n(dt) G_n(|Z_0|,|Z_1|,...,|Z_{n-1}|,|z|)dz.$$

\subsection{Gauge symmetries of L\'evy processes}\label{subsection_gauge_Levy}

Generalizing \cite[Chapter II Definition 4.1 and Theorem 4.15]{Jacod2003} we introduce the following definition.

\begin{definition}\label{definition_Levy}
A c\acc{a}dl\acc{a}g semimartingale $Z$ on a Lie group $N$ is an \emph{independent increments process} if its characteristics $(b,A,\nu)$ are deterministic.\\
The process $Z$ is a L\'evy process if $b_t=b_0t,A_t=A_0 t,\nu(dt,dx)=\nu_0(dx)dt$ for some $b_0 \in \mathbb{R}^n$, $A_0$ $n \times n$ symmetric
 positive semidefinite matrix and some $\sigma$-finite measure $\nu_0$ on $N$ such that $\int_N{(h^{\alpha}(z))^2\nu_0(dz)}<+\infty$ and
$\int_N{f(z)\nu_0(dz)}<+\infty $ for any smooth and bounded function $f \in \cinf(N)$ which is identically zero in a neighbourhood of $1_N$.
\end{definition}

Definition \ref{definition_Levy} is equivalent to the concept with the same name proposed by Feinsilver \cite{Feinsilver1978} (more recently studied by Liao \cite{Liao2014}) with the further request that the process $Z$ is a semimartingale. The definition of independent increments process depends on the filtration $\mathcal{F}_t$ used for defining the characteristics $(b,A,\nu)$ and, since $(b,A,\nu)$ are deterministic, the filtration $\mathcal{F}_t$
has to be a generalized natural filtration.

\begin{remark}
The characteristics of a L\'evy process introduced in Definition \ref{definition_Levy} are the same as those discussed in Subsection \ref{subsubsection_Levy}.
\end{remark}

\begin{theorem}\label{theorem_characteristic3}
If a semimartingale $Z$ is an independent increments process such that its law is uniquely determined by its characteristics, and if $\mathcal{G}$ is a metrizable second countable topological group, then $Z$ admits
$\mathcal{G}$ as  gauge symmetry group  with action $\Xi_g$ if and only if, for any $g \in \mathcal{G}$, \normalsize
\begin{equation}
\label{equation_characteristic5}
b^{\alpha}_t=\Gamma^{\alpha}_{g,\beta}b_t^{\beta}+\frac{1}{2}O^{\alpha}_{g,\beta\gamma}A^{\beta\gamma}
+\int_0^t{\int_N{(h^{\alpha}(z')-h^{\beta}(\Xi_{g^{-1}}(z'))\Gamma^{\alpha}_{g,\beta})\nu(ds,dz')}}
\end{equation}
\begin{eqnarray}
A^{\alpha\beta}_t&=&\Gamma^{\alpha}_{g,\gamma}\Gamma^{\beta}_{g,\delta}A^{\gamma\delta}_t\label{equation_characteristic6}\\
\nu&=&\Xi_{g*}(\nu).\label{equation_characteristic7}
\end{eqnarray}
\normalsize
\end{theorem}
\begin{proof}
Let us consider the constant process $G_t=g_0$ for some $g_0 \in \mathcal{G}$. Since $\Xi_{g_0}$ is a diffeomorphism and since the constant
process $G_t=g_0$ is measurable with respect to both the natural filtrations of $Z_t$ and of $\tilde{Z}_t$, it is simple to prove that, if
$\tilde{\mathcal{F}}^c_t$ is a generalized natural
 filtration for $Z_t$, then it is a generalized natural filtration also for $d\tilde{Z}_t=\Xi_{g_0}(dZ_t)$. This fact implies that $\hat{\mathcal{F}}^c_t$ is a generalized natural filtration
for $\omega_A(t)$ with respect to the law $\mathbb{P}'$ (where $\hat{\mathcal{F}}^c_t, \omega_A(t), \mathbb{P}'$ were defined in Section \ref{subsection_gauge_main}).
For this reason, and since  $(b,A,\nu)$ and the process $G_t$
  do not depend on $\omega$, \refeqn{equation_characteristic5},
\refeqn{equation_characteristic6} and \refeqn{equation_characteristic7} follow from the necessary condition in Theorem \ref{theorem_characteristic1}.  \\
Conversely, if equations \refeqn{equation_characteristic5}, \refeqn{equation_characteristic6} and \refeqn{equation_characteristic7} hold, they
imply equations \refeqn{equation_characteristic8}, \refeqn{equation_characteristic9} and \refeqn{equation_characteristic10} for any elementary
process $G_t$. Using the density of elementary processes in $\mathfrak{G}$ (the space of locally bounded predictable processes) when $\mathcal{G}$ is metrizable and second countable (see for example \cite[Lemma 3.2.6]{Krylov1980}) and exploiting the dominated convergence theorem for It\^o stochastic integration (see for example \cite[Chapter IV Theorem 32]{Protter1990}) we can extend \refeqn{equation_characteristic8}, \refeqn{equation_characteristic9}
and \refeqn{equation_characteristic10} to the case of any locally bounded predictable process $G_t$.\\
Since the law of $Z$ is uniquely determined by its characteristics, the thesis follows  by the sufficient condition in Theorem
\ref{theorem_characteristic2}.
\hfill\end{proof}

\noindent When $\mathcal{G} \subset Aut(N)$ is the Lie group of the smooth automorphisms of $N$, and $\Xi_g$ is the natural action of $Aut(N)$ on $N$ the stochastic invariance is equivalent to the deterministic one.

\begin{corollary}\label{corollary_Levy}
If $\mathcal{G} \subset Aut(N)$ and $\Xi_g$ is the natural action of $Aut(N)$, an independent increments process $Z$ admits the gauge symmetry group $\mathcal{G}$ with respect to any generalized natural filtration $\mathcal{F}_t$ if and only if, for any $g \in \mathcal{G}$, $Z'_t=\Xi_g(Z_t)$ has the same law of $Z$.
\end{corollary}
\begin{proof}
The corollary is an easy consequence of Lemma \ref{lemma_characteristic1},  Theorem \ref{theorem_characteristic3} and Proposition \ref{proposition_automorphism}.
\hfill\end{proof}\\

\noindent Corollary \ref{corollary_Levy} shows that gauge symmetries provide a non-trivial generalization of the concept of deterministic invariance of L\'evy processes. In particular rotation-invariant L\'evy processes on $N=\mathbb{R}^n$ are also invariant with respect to random rotations. Indeed, when $\mathcal{G}=O(n)$ and $\Xi_B(z)=B \cdot z$, we have that $d\tilde{Z}_t=\Xi_{B_t}(dZ_t)$ satisfies
$$Z'^{\alpha}_t=\int_0^t{B^{\alpha}_{\beta,s} dZ^{\beta}_s}.$$
Examples of such processes are L\'evy processes with generator of the form
$$\begin{array}{rcl}
L(f)(z)&=&\sum_{\alpha=1}^n\frac{D}{2}\partial_{z^{\alpha}z^{\alpha}}(f)(z)+\\
&&+\int_{N}{(f(z+z')-f(z)-I_{|z'|<1}(z')z^{\alpha}\partial_{z^{\alpha}}(f))F(|z'|)dz'},\end{array}$$ \normalsize where $D \in \mathbb{R}_+$,
$|\cdot |$ is the standard norm on $\mathbb{R}^n$ and $F:\mathbb{R}_+  \rightarrow \mathbb{R}_+$ is a measurable locally bounded function such
that $\int_1^{\infty}{F(r)r^{n-1}dr} < + \infty$ and $\int_0^1{F(r)r^{n+1}} < +\infty$ and in particular Brownian motion (for $D=1$ and $F=0$) and $\alpha$-stable processes (for $D=0$ and $F(|z|)=\frac{1}{|z|^{n+\alpha}}$).

\subsection{Examples of gauge symmetric non-Markovian semimartingales}\label{subsection_non_markovian}

For general non-Markovian semimartingales, where $(b,A,\nu)$ depends explicitly on $\omega$, we cannot simplify Theorem \ref{theorem_characteristic2} as in the case of a L\'evy process. Nevertheless there is a special case where it is possible to use a strategy similar to the one proposed in Theorem \ref{theorem_characteristic3}.\\
\begin{theorem}\label{theorem_characteristic4}
Let $N=N_1 \times N_2$ (where $N_1,N_2$ are two Lie groups) be a Lie group with multiplication defined by
$$(z_1,z_2) \cdot (z'_1,z'_2)= (z_1 \cdot_1 z'_1, z_2 \cdot_2 z'_2),$$
where   $\cdot_1,\cdot_2$ denote the multiplication on
$N_1,N_2$, respectively. If we take the space $\Omega_A=\Omega_A^1 \times \Omega_A^2$, where $\Omega_A^i=\mathcal{D}_{1_{N_i}}([0, + \infty),N_i)$,
and we denote by $\omega_A^1,\omega^2_A$ the elements of $\Omega_A^1,\Omega_A^2$, respectively, we can consider $\Xi_g=(\Xi^1_g,id_{N_2})$. If the characteristics of a semimartingale $Z$ in $N$ depend only on $\omega_A^2$ and $Z$ admits the Lie group $\mathcal{G}$ with action $\Xi_g$ as a gauge symmetry group then, for any $g \in \mathcal{G}$,
 \normalsize

\begin{equation}\label{equation_characteristic11}
\begin{array}{rcl}
b^{\alpha}_t(\omega^2_A)&=&\Gamma^{\alpha}_{g,\beta}b_t^{\beta}(\omega^2_A)+\frac{1}{2}O^{\alpha}_{g,\beta\gamma}A^{\beta\gamma}(\omega^2_A)+\\
&&+\int_0^t{\int_N{(h^{\alpha}(z')-h^{\beta}(\Xi_{g^{-1}}(z'))\Gamma^{\alpha}_{g,\beta})\nu(\omega^2_A,ds,dz')}}\end{array}
\end{equation}
\begin{eqnarray}
A^{\alpha\beta}_t(\omega^2_A)&=&\Gamma^{\alpha}_{g,\gamma}\Gamma^{\beta}_{g,\delta}A^{\gamma\delta}_t(\omega^2_A)\label{equation_characteristic12}\\
\nu(\omega^2_A,dt,dz)&=&\Xi_{g*}(\nu(\omega^2_A,dt,dz)).\label{equation_characteristic13}
\end{eqnarray}
\normalsize Moreover, if the triplet $(b,A,\nu)$ uniquely determines the law of $Z$ and the group $\mathcal{G}$ is metrizable and second countable, then equations \refeqn{equation_characteristic11},
\refeqn{equation_characteristic12} and \refeqn{equation_characteristic13} ensure that $Z$ admits $\mathcal{G}$ as gauge symmetry group.
\end{theorem}
\begin{proof}
The proof is based on Theorem \ref{theorem_characteristic2}
and on the fact that the map $\Lambda'$ appearing in  Theorem \ref{theorem_characteristic2}, in this situation has the form
$$\Lambda'=\left(\begin{array}{c}
\Lambda'^1_A\\
id_{\Omega_A^2}\\
id_{\Omega_B}\end{array} \right).$$
In particular, the necessary condition can be proved by considering  the constant process $G_t=g_0$ and applying Theorem \ref{theorem_characteristic2}, while
the proof of the sufficiency of equations
\refeqn{equation_characteristic11},
\refeqn{equation_characteristic12} and
\refeqn{equation_characteristic13} is  similar to
the proof of Theorem \ref{theorem_characteristic3}.
\hfill\end{proof}\\

\begin{remark}\label{remark_non_markovian}
Theorem \ref{theorem_characteristic4} in some way provides a simplification of Theorem \ref{theorem_characteristic2} when the action $\Xi_g$ has a special form but can also be exploited in a different way.\\
For example, let $Z$ be a semimartingale on a Lie group $N_1$ such that there exists a semimartingale $Z'$ on a Lie group $N_2$ with the property that the semimartingale $(Z,Z') \in N_1 \times N_2$ admits characteristics $(b,A,\nu)$ (with respect to their natural filtration) of the form required for Theorem \ref{theorem_characteristic4}. If the characteristic triplet $(b,A,\nu)$ uniquely determines the law of the process $(Z,Z')$ and the topological group $\mathcal{G}$ is locally metrizable and second countable, then $\mathcal{G}$ with action $\Xi_g$ is a gauge symmetry group for $Z$ with respect to its natural filtration if and only if, for any $g \in \mathcal{G}$, $(\Xi_g(dZ),dZ')$ has the same law as $(dZ,dZ')$. If $\mathcal{G} \subset Aut(N)$, by Proposition \ref{proposition_automorphism}, this is equivalent to the request that $(\Xi_g(Z),Z')$ has the same law of $(Z,Z')$.
\end{remark}

We apply Remark \ref{remark_non_markovian} to an explicit example, constructing an $\mathbb{R}^2$-semi-martingale admitting $O(2)$ with its natural action $\Xi_B(z)=B \cdot z$ as gauge symmetry group. Let $(W^1,W^2,W^0)$ be a three dimensional Brownian motion and consider
$$\tilde{W}^{\alpha}_t=\int_0^t{G(W^0_{[0,s]})dW^{\alpha}_s}.$$
The characteristics of the $\mathbb{R}^3$ semimartingale $(\tilde{W}^1,\tilde{W}^2,W^0)$ are
\normalsize
\begin{eqnarray*}
db_t&=&0\\
dA_t&=&\left(\begin{array}{ccc}
(G(W^0_{[0,t]},t))^2dt & 0 & 0 \\
0 & (G(W^0_{[0,t]},t))^2dt & 0 \\
0 & 0 & dt
\end{array} \right)\\
\nu&=&0,
\end{eqnarray*}
\normalsize
where the functions $h^{\alpha}$ can be chosen arbitrarily.\\
Since $A$ is invariant with respect to rotations in the first two components, $(B\cdot \tilde{W},W^0)$ has the same law as $(\tilde{W},W^0)$ for any $B \in O(2)$ (where $\tilde{W}=(\tilde{W}^1 ,\tilde{W}^2) \in \mathbb{R}^2$). By Remark \ref{remark_non_markovian} $(\tilde{W}^1,\tilde{W}^2)$ has $O(2)$ as a gauge symmetry group with respect to the natural filtration of $(\tilde{W}^1,\tilde{W}^2)$.

\begin{remark}
The previous example is the prototypical example of a continuous semimartingale on $\mathbb{R}^n$ admitting $O(n)$, with its natural action, as a gauge symmetry group. Indeed in \cite{DeVecchi2017(deFinetti)} it is proven, with few technical hypotheses on the characteristics $(b,A,\nu)$, that a semimartingale $Z$ with continuous paths, which admits $O(n)$ as a gauge symmetry group, is of the form $Z_t=\int_0^t{f_t dW_t}$, where $W_t$ is an $n$ dimensional Brownian motion and $f_t$ is a stochastic process  independent of $W$ and with $\int_0^t{f^2_sds} < + \infty$ almost surely.
\end{remark}

\section*{Acknowledgements}
The first author would like to thank the Department of Mathematics, Universit\`a degli Studi di Milano for the warm hospitality. This research was supported by IAM and HCM (University of Bonn) and Gruppo Nazionale Fisica Matematica (GNFM-INdAM). The second author is supported by the German Research Foundation (DFG) via CRC 1060. 

\bibliographystyle{plain}

\bibliography{levy_symmetries(8)}

\end{document}